\documentclass[a4paper, 11pt]{article}
\usepackage{anysize}
\marginsize{3,5cm}{2,5cm}{2,5cm}{2,5cm}
\usepackage{amssymb}
\usepackage{amsmath,amsbsy,amssymb,amscd}
\usepackage{t1enc}
\usepackage[cp1250]{inputenc}
\usepackage[british]{babel}
\usepackage[all]{xy}
\usepackage{color}
\usepackage{amsfonts}
\usepackage{latexsym}
\usepackage{amsthm}
\usepackage{mathrsfs}
\usepackage{hyperref}
\usepackage{graphicx}
\usepackage{comment}

\usepackage{color}
\usepackage{caption}
\usepackage{subcaption}
\usepackage{enumerate}
\usepackage{xcolor}
\usepackage[draft]{todonotes}
\usepackage{float}

\definecolor{orange}{rgb}{1,0.5,0}

\usepackage{mathrsfs} 

\DeclareMathAlphabet{\mathpzc}{OT1}{pzc}{L}{it} 

\theoremstyle{definition}
\newtheorem{definition}{Definition}[section]
\newtheorem{theorem}[definition]{Theorem}

\newtheorem{corollary}[definition]{Corollary}
\newtheorem{lemma}[definition]{Lemma}

\newtheorem{remark}[definition]{Remark}

\newtheorem{claim}[definition]{Claim}
\newtheorem*{lm:mainLemma}{Theorem \ref{lm:mainLemma}}

\def\geq{\geqslant}
\def\leq{\leqslant}
\def\R{\mathbb{R}}

\def\id{\mathrm{id}}

\def\epsilon{\varepsilon}

\newcommand{\bea}{\begin{eqnarray}}
  \newcommand{\eea}{\end{eqnarray}}
  \newcommand{\beab}{\begin{eqnarray*}}
  \newcommand{\eeab}{\end{eqnarray*}}

  \newcommand{\be}{\begin{equation}}
  \newcommand{\ee}{\end{equation}}

\title{Kakutani equivalence for products of some special flows over rotations\\
\scriptsize\emph{To the memory of my dearest teacher Prof. Anatole Katok whose careful guidance and warm advice will always stay in my heart.} }
\author{Daren Wei\footnote{Einstein Institute of Mathematics, The Hebrew University of Jerusalem, Givat Ram. Jerusalem, 9190401, Israel, E-mail: Daren.Wei@mail.huji.ac.il  D. W. was partially supported by the NSF grant DMS-16-02409 and ERC-2018-ADG project HomDyn.}}

\begin{document}
\maketitle
\begin{abstract}
We study Kakutani equivalence for products of some special flows over rotations with roof function smooth except a singularity at $0\in\mathbb{T}$. We estimate the Kakutani invariant for product of these flows with different powers of singularities and rotations from a full measure set. As a corollary, we obtain a countable family of pairwise non-Kakutani equivalent products of special flows over rotations.
\end{abstract}

\section{Introduction}

One of the central problems in ergodic theory is the classification of ergodic m.p.t. (measure preserving transformations). The first attempt was made by considering the \emph{measurably isomorphism equivalence}: two m.p.t. are measurably isomorphic if there is a measure preserving one-to-one conjugate map between them. But due to \cite{BFor}, \cite{FWeiss} and \cite{ForRudWeiss}, this problem is too complex in general. A possible solution to obtain a complete answer for the classification problem of ergodic m.p.t. is to consider some weaker equivalence relations than measurably isomorphism equivalence. An \emph{orbit equivalence} is a natural candidate that has this feature: two m.p.t are called orbit equivalent if there exists a measure preserving invertible map between their orbits (i.e. map orbits as sets and does not preserve the order). However, this equivalence relation is trivial among ergodic m.p.t. due to Dye's theory \cite{Dye1}, \cite{Dye2}. Around 1940s, S. Kakutani \cite{Kak} introduced a new equivalence relation: for ergodic m.p.t. $T$ defined on $X$ and $S$ defined on $Y$, they are \emph{Kakutani equivalent} if there exist $A\subset X$ and $B\subset Y$ such that $T|_A$ and $S|_B$ are measurably isomorphic (see Definition \ref{def:Kakutani} for more details). As Kakutani equivalence is weaker than isomorphism equivalence but stronger than orbit equivalence, it is natural to consider classification of ergodic m.p.t. based on this equivalence relation.

In order to measure the complexity among Kakutani equivalence classes, A. Katok \cite{Katok1}, \cite{Katok2} established a binary relation, which is defined as \emph{majorization}: an automorphism $T$ majorizes an automorphism $S$ if there exists an automorphism $T_1$ which is Kakutani equivalent to $T$ and an invariant partition $\xi$ such that $T_1|_{\xi}$ is Kakutani equivalent to $S$. This binary relation introduces a quasi-partial order (without antisymmetry) structure among Kakutani equivalence classes. Among all the Kakutani equivalence classes, there is a special Kakutani equivalence class which defined as \emph{standard automorphism}\footnote{Also called as zero entropy loosely Bernoulli or loosely Kronecker.}: an automorphism is called standard if it is Kakutani equivalent to an irrational rotation. Standard automorphisms are special as they are majorized by any ergodic automorphisms (\cite{Katok2}) and thus can be understood as the Kakutani equivalence class which has the minimal complexity. Standardness can also be defined for flows:  a flow is a \emph{standard flow} if it is Kakutani equivalent to a linear flow $R_t(x,y)=(x+t,y+t\alpha)$ on $\mathbb{T}^2$ for $\alpha\in\mathbb{R}\setminus\mathbb{Q}$.

Later, inspired by the idea that metric entropy is an invariant of isomorphism equivalence, M. Ratner \cite{Ratner3}, \cite{Ratner4} created an invariant of Kakutani equivalence similar to the metric entropy and called it \emph{Kakutani invariant} and denoted by $e(T_t,u)$, where $T_t$ is the flow and $u$ is the scaling function (see Section \ref{sec:KakutaniInvariant} for more details). As metric entropy measures the complexity of automorphisms among isomorphic equivalence classes, the Kakutani invariant also measures the complexity of automorphisms among the Kakutani equivalence classes. Moreover, the way that the Kakutani invariant measures the complexity coincides with majorization in some sense \cite{Ratner3}: an automorphism has zero Kakutani invariant at all scales if and only if it is a standard automorphism.

In fact, standardness is a quite stable property as it is closed under factors, inverse limits and compact extensions \cite{Katok2}, \cite{OrnRudWei}, \cite{Bron}. Moreover, there are many natural examples of standard systems: all systems of local rank one \cite{Ferenczi}, all distal systems, and in particular, all nil-systems. Kakutani originally even conjectured that all zero entropy systems were standard\footnote{Although he did not use this terminology.} \cite{Kak} although this is not the case. The first such examples were constructed by the cutting and stacking method by J. Feldman \cite{Feldman} and the first non-standard smooth examples on smooth manifolds were constructed in \cite{Katok2} by A. Katok. Later, D. S. Ornstein, B. Weiss and D. Rudolph also constructed some other non-standard examples in \cite{OrnRudWei}.  In \cite{Ratner1}, M. Ratner gave a natural algebraic example of a non-standard system in dimension $6$ in the smooth category. More precisely, Ratner showed that the product of the horocycle flow on compact homogeneous space is not standard (the horocycle flow itself being standard, \cite{Ratner2}). In \cite{KanigowskiWei1}, Kanigowski and the author proved non-standardness of product of Kochergin flows with different exponents and thus gave natural examples of non-standard smooth flows in dimension $4$. It is worth to notice that there are examples of even smooth zero entropy diffeomorphisms on any compact manifold admitting an effective smooth $\mathbb{T}^2$-action whose Cartesian product of itself is loosely Bernoulli \cite{GK}. In \cite{KanigowskiVinhageWei2}, Kanigowski, Vinhage and the author showed the only ergodic unipotent flows on finite volume quotients of linear semisimple Lie groups which are standard are of the form $\phi_t=\begin{pmatrix}1&t\\0&1\end{pmatrix}\times \id$ acting on $(\operatorname{SL}(2,\R)\times G')\slash \Gamma$,
where $\Gamma$ is irreducible. This result also provide many non-standard algebraic flows on homogeneous space as it fully describes the structure of standard unipotent flows on homogeneous space. Recently, Kanigowski and De la Rue \cite{KanigowskiDeLaRue} showed non-standardness of product of two staircase rank one transformations, which provided a non-standard example among products of natural class of mixing rank one transformations.

The aim of this paper is to study the Kakutani invariant of products of some special flows: special flows over irrational rotations with one degenerate fixed point and roof function in the form $x^{\gamma}$ for $\gamma\in(-1,0)$. These types of special flows are related to special flows coming from smooth $\mathbb{T}^2$ flows: Kochergin flows are a special case of these types of special flow with exponents $\gamma$ in the form $-(1-\frac{1}{n})$ for integers $n\geq2$. Moreover, these flows are similar to horocycle flows as they are standard, mixing \cite{Koch}, some are mixing of all orders \cite{KanigowshiFayad}, almost every have countable Lebesgue spectrum \cite{FFK} and polynomial orbit growth \cite{Kanigowski}. In this paper, we estimate the Kakutani invariant of products of these flows with different exponents and thus obtain a countable family of pairwise non-Kakutani equivalent products of non-smooth special flows over rotations.

\subsection{Statement of Main Results}

In order to state our main results  more precisely, we will introduce several notation. The special flows under consideration are flows over irrational rotations $T_{\alpha}$, $\alpha\in(\mathbb{R}\backslash\mathbb{Q})\cap(\mathbb{R}/\mathbb{Z})$, with roof functions in $C^2(\mathbb{T}\backslash\{0\})$ which have similar asymptotic behavior around $0$ as $x^{\gamma}$, $-1<\gamma<0$ (see Section \ref{sec:KocherginFlows} for a precise definition). In this setting, every such flow is given by a pair $(\alpha,\gamma)\in[(\mathbb{R}\backslash\mathbb{Q})\cap(\mathbb{R}/\mathbb{Z})]\times(-1,0)$ and thus we denote such special flow by $(\mathscr{T}^{\alpha,\gamma}_t)$. Let $(q_{\alpha,n})_{n\geq 1}$ represent the sequence consisting of denominators of convergents of continued fractions of an irrational number $\alpha$ and
\begin{equation}\label{def:setD}
\mathscr{D}:=\{\alpha\in(\mathbb{R}\backslash\mathbb{Q})\cap(\mathbb{R}/\mathbb{Z}):q_{\alpha,n+1}<C(\alpha)q_{\alpha,n}\log q_{\alpha,n}(\log n)^2\},
\end{equation}
it follows from Khinchin's theorem [\cite{Kinchin}, Theorem $30$, p.$63$] that $\lambda(\mathscr{D})=1$, where $\lambda$ is the Lebesgue measure on $\mathbb{T}$. Besides the full measure property, we would also like to mention another very useful feature of the set $\mathscr{D}$. Notice that for any $\alpha\in\mathscr{D}$ and any $N\in\mathbb{N}$, there exists an integer $n$ such that $N\in[q_{\alpha,n},q_{\alpha,n+1})$. By the definition of $\mathscr{D}$, i.e. \eqref{def:setD}, there exists $N_{\alpha}\in\mathbb{N}$ such that if $N\geq N_{\alpha}$, then $\log q_{\alpha,n}>\max\{1000,C(\alpha)\}$, where $C(\alpha)$ is from the definition of set $\mathscr{D}$. Combining the definition of set $\mathscr{D}$, the definition of the best approximation and Khinchin's Theorem [\cite{Kinchin}, Theorem $13$, p.$15$], we obtain the following for $N\geq N_{\alpha}$:
\begin{equation}\label{eq:Ncontrol}
\begin{aligned}
\min_{m\in\mathbb{Z}}\{|N\alpha-m|\}&\geq\frac{1}{q_{\alpha,n+1}+q_{\alpha,n+2}}\\
&\geq\frac{1}{q_{\alpha,n}\log^4q_{\alpha,n}+q_{\alpha,n}\log^9q_{\alpha,n}}\\
&\geq\frac{1}{q_{\alpha,n}\log^{10}q_{\alpha,n}}\geq\frac{1}{N\log^{10}N}.
\end{aligned}
\end{equation}


The flows $(\mathscr{T}^{\alpha,\gamma}_t)$ under considerations are the special flows over irrational rotations $T_{\alpha}:\mathbb{T}\to\mathbb{T}$ such that  $T_{\alpha}x=x+\alpha\mod 1$ ($\alpha\in[(\mathbb{R}\backslash\mathbb{Q})\cap(\mathbb{R}/\mathbb{Z})]$) with the roof functions $f\in C^2(\mathbb{T}\backslash\{0\})$ satisfying the following conditions for some $-1<\gamma<0$ and $A_1,B_1>0$:
\begin{equation}\label{eq:powerBehavior}
\begin{aligned}
\lim_{x\to0^+}\frac{f(x)}{x^{\gamma}}=A_1 \ \ &\text{and}\ \ \lim_{x\to0^-}\frac{f(x)}{(1-x)^{\gamma}}=B_1,\\
\lim_{x\to0^+}\frac{f'(x)}{x^{-1+\gamma}}=\gamma A_1 \ \ &\text{and}\ \ \lim_{x\to0^-}\frac{f'(x)}{(1-x)^{-1+\gamma}}=-\gamma B_1,\\
\lim_{x\to0^+}\frac{f''(x)}{x^{-2+\gamma}}=\gamma(\gamma-1)A_1\ \ &\text{and}\ \ \lim_{x\to0^-}\frac{f''(x)}{(1-x)^{-2+\gamma}}=\gamma(\gamma-1)B_1.
\end{aligned}
\end{equation}

Now we can formulate our main results as following:

\begin{theorem}\label{thm:Main}
For $\alpha_1,\alpha_2\in\mathscr{D}$ and every $\gamma_1,\gamma_2\in(-1,0)$ with $|\gamma_1|>|\gamma_2|$, we have
$$\frac{2+4|\gamma_2|+2|\gamma_1\gamma_2|}{(2+|\gamma_2|)(1+|\gamma_1|)}\leq e((\mathscr{T}^{\alpha_1,\gamma_1}\times\mathscr{T}^{\alpha_2,\gamma_2})_t,\log)\leq1+|\gamma_1|+|\gamma_2|.$$
\end{theorem}

By applying Theorem $1$ in \cite{Ratner3}, which gives the equivalence between standardness and vanishing of the Kakutani invariant at all scales for any zero entropy ergodic flows, we obtain the main result of \cite{KanigowskiWei1} as a natural corollary of Theorem \ref{thm:Main}:

\begin{corollary}\label{Cor:Nonstand}
For $\alpha_1,\alpha_2\in\mathscr{D}$ and every $\gamma_1, \gamma_2\in(-1,0)$ with $\gamma_1\neq\gamma_2$, the corresponding product of Kochergin flows with different exponents is not standard.
\end{corollary}

\begin{remark}
Let's point out that the fact $\gamma_1\neq\gamma_2$ is essential for our proof of Theorem \ref{thm:Main} and thus we cannot extend our result to this case. More precisely, if $\gamma_1=\gamma_2$, we cannot obtain the different shearing rates in different copies of product system (see \eqref{eq:ChoiceOfEpsilon0}, \eqref{eq:restrictionChoice1} and \eqref{eq:restrictionChoice2}). As a result, we can't even obtain the non-standardness of the product system, which is also the case that \cite{KanigowskiWei1}'s result does not include.
\end{remark}

In fact, we can obtain a countable family of pairwise non-Kakutani equivalent products of non-smooth special flows over rotations by Theorem \ref{thm:Main}.

The idea is to repeatedly use Theorem \ref{thm:Main} and to guarantee that for different choices of $\gamma_1$ and $\gamma_2$ the intervals formed by lower and upper bound in Theorem \ref{thm:Main} are pairwise disjoint. More precisely, notice that $|\gamma_2|>\frac{2|\gamma_1|}{3+|\gamma_1|}$ implies $\frac{2+4|\gamma_2|+2|\gamma_1\gamma_2|}{(2+|\gamma_2|)(1+|\gamma_1|)}>1$, then for $\alpha_1,\alpha_2,\alpha'_1,\alpha'_2\in\mathscr{D}$, $-1<\gamma_1<\gamma_2<\frac{2\gamma_1}{3+|\gamma_1|}<0$ and $-1<\gamma'_1<\gamma'_2<\frac{2\gamma'_1}{3+|\gamma'_1|}<0$, it is possible to pick $\gamma'_1,\gamma'_2$ such that $1+|\gamma'_1|+|\gamma'_2|<\frac{2+4|\gamma_2|+2|\gamma_1\gamma_2|}{(2+|\gamma_2|)(1+|\gamma_1|)}$. For simplicity, we define $\Delta(\gamma_1,\gamma_2)=\left[\frac{2+4|\gamma_2|+2|\gamma_1\gamma_2|}{(2+|\gamma_2|)(1+|\gamma_1|)},1+|\gamma_1|+|\gamma_2|\right]$. Then the choice of $\gamma_1'$ and $\gamma_2'$ implies $\Delta(\gamma_1',\gamma_2')\cap\Delta(\gamma_1,\gamma_2)=\emptyset$ and $\Delta(\gamma_1',\gamma_2')$ is on the left side of $\Delta(\gamma_1,\gamma_2)$. By repeating this procedure for $\gamma_1'$ and $\gamma_2'$, we obtain that there exist $\gamma_1''$ and $\gamma_2''$ such that $\Delta(\gamma_1'',\gamma_2'')\cap\Delta(\gamma_1',\gamma_2')=\emptyset$ and $\Delta(\gamma_1'',\gamma_2'')$ is on the left side of $\Delta(\gamma_1',\gamma_2')$.  By continuing this procedure, we can get countably many pairwise non-intersecting intervals $\Delta(\gamma_1^i,\gamma_2^i)$ for $i=1,2,\ldots$. Then we apply Theorem $3$ in \cite{Ratner3} and Corollary $1$ in \cite{Ratner4}, we obtain the following corollary:
\begin{corollary}\label{Cor:Uncoutable}
There exists a countable family of pairwise non-Kakutani equivalent products of non-smooth special flows over rotations.
\end{corollary}

\begin{remark}
The flows constructed in Corollary \ref{Cor:Uncoutable} are not smooth in the current setting. More precisely, the smoothness of Kochergin flows comes from the selection of the Hamiltonian functions, which are smooth in case that the exponents are of the form $-(1-\frac{1}{n})$ for $n\in\mathbb{N}$ (see the last section of \cite{Koch} for more details). Recall that in the proof of Corollary \ref{Cor:Uncoutable}, the absolute value of exponents are decreasing extremely fast and thus exponents are larger than $-\frac{1}{3}$ besides the first pair, which causes that these examples do not represent smooth flows. We strongly believe that more precise estimates of the Kakutani invariant will imply existence of countably many non-Kakutani equivalent smooth Kochergin flows. It is also worth to notice that as the smoothness comes from the Hamiltonian functions, we can obtain at most countably many non-Kakutani equivalent smooth Kochergin flows due to Hamiltonian functions' exponents' selections.
\end{remark}

\begin{remark}
Corollary \ref{Cor:Uncoutable} can be compared with Remark $2.5$ in \cite{KanigowskiWei1}, where uncountably many non-isomorphic (non-isomorphic by \cite{Kanigowski}) non-standard smooth flows in dimension $4$ are obtained. Recall that two flows may not be measurably isomorphic but Kakutani equivalent, thus this result does not imply the existence of uncountably many non-Kakutani equivalence classes in products of special flows.
\end{remark}

\begin{remark}
This Corollary should also be compared with Benhenda \cite{Benhenda}, where the author constructed uncountably many non-Kakutani equivalent smooth diffeomorphisms by an AbC method. Our result provides countably many non-Kakutani equivalent natural special flows.
\end{remark}

\textbf{Plan of the paper:} In Section \ref{sec:BDP} we introduce several basic definitions, notation and lemmas: special flows, flows under consideration, Kakutani invariant, Denjoy-Koksma inequality and some ergodic sums estimates. In Section \ref{sec:mainThm}, we estimate the Kakutani invariant of the system based on some ergodic estimates and Theorem  \ref{lm:mainLemma}, which describes a relation between a good matching and the metric of $\mathbb{T}^{f_1}\times\mathbb{T}^{f_2}$. In Section \ref{sec:proofOfMainLemma} and Section \ref{sec:proofOfSecondLemma}, we prove Theorem \ref{lm:mainLemma} by some combinatorial techniques and Lemma \ref{lm:thirdMain}. In Section \ref{sec:proofOfThirdLemma}, we prove Lemma \ref{lm:thirdMain} based on some observations of good matching of special flows.

\textbf{Acknowledgements.} The author would like to thank his advisor Svetlana Katok for her help, support and warm encouragements. The author would also like to thank Adam Kanigowski for suggesting this project, many useful and helpful discussions. The author is grateful to Changguang Dong and Philipp Kunde for careful reading on the first draft of the paper. The author is deeply grateful for referee's careful reading, useful comments and insightful suggestions about how to improve the paper.

\section{Basic definitions and propositions}\label{sec:BDP}
In this section we will recall the definitions of Kakutani equivalence, Kakutani invariants, flows under considerations, Denjoy Koksma inequality and some estimates of ergodic sums.

\subsection{Kakutani invariant}\label{sec:KakutaniInvariant}
We first recall the definition of Kakutani equivalence. For a flow $(T_t)$ on $(X,\mathscr{B},\mu)$ and a positive function $\alpha\in L^1_+(X,\mathscr{B},\mu)$\footnote{Recall that $f\in L^1_+(X,\mathscr{B},\mu)$ if $f\in L^1(X,\mathscr{B},\mu)$ and $f(x)>0$ for any $x\in X$.}, we define the time change $T_t^{\alpha}(x)=T_{u(x,t)}(x)$, where $u(x,t)$ is a solution to $\int_0^{u(x,t)}\alpha(T_sx)ds=t$. By this construction, $(T_{t}^{\alpha})$ preserves the measure $d\hat{\mu}=\frac{\alpha(\cdot)}{\int_X\alpha d\mu}d\mu$.
\begin{definition}\label{def:Kakutani}
Two ergodic measure preserving flows $(X,(T_t),\mathscr{B},\mu)$ and $(Y,(S_t),\mathscr{C},\nu)$ are \emph{Kakutani (monotone) equivalent} if there exists a time change $\alpha\in L^1_+(X,\mathscr{B},\mu)$ such that $(S_t)$ is measurably isomorphic to $(T_t^{\alpha})$.
\end{definition}

We will follow \cite{Ratner3} and \cite{Ratner4} to define the Kakutani invariant. Suppose $(T_t)$ is a measure preserving flow on a probability space $(X,\mathscr{B},\mu)$. Let $\mathcal{P}$ be a finite measurable partition of $X$ and $\mathcal{P}(x)$ be the atom of $\mathcal{P}$ containing $x\in X$.

\begin{definition}[$(\epsilon,\mathcal{P})-$matchable, \cite{Ratner4}]\label{def:ematchable}
Let $I_R(x)$ be the orbit interval $[x,T_Rx]$ and $\bar{\lambda}$ be the Lebesgue measure on $[0,R]$. For $x,y\in X$, $\epsilon>0$ and $R>1$, $I_R(x)$ and $I_R(y)$ are called $(\epsilon,\mathcal{P})-$matchable if there exists a subset $A=A(x,y)\subset[0,R]$, $\bar{\lambda}(A)>(1-\epsilon)R$ and an increasing absolutely continuous map $h=h(x,y)$ from $A$ onto $A'=A'(x,y)\subset[0,R]$, $\bar{\lambda}(A')>(1-\epsilon)R$ such that $\mathcal{P}(T_tx)=\mathcal{P}(T_{h(t)}y)$ for all $t\in A$ and the derivative $h'=h'(x,y)$ satisfies
$$|h'(t)-1|<\epsilon\text{ for all }t\in A,$$
where $h$ is defined as an $(\epsilon,\mathcal{P},R)-$\textrm{matching} from $I_R(x)$ onto $I_R(y)$.
\end{definition}

We provide the following picture to give an intuitive explanation of $(\epsilon,\mathcal{P},R)-$matching, where the lines after $x$ and $y$ represent their orbits respectively; different colors represent the orbit segments that stay in distinct atoms of the partition $\mathcal{P}$; dotted lines connect the matching pairs on the orbit of $x$ and orbit of $y$:
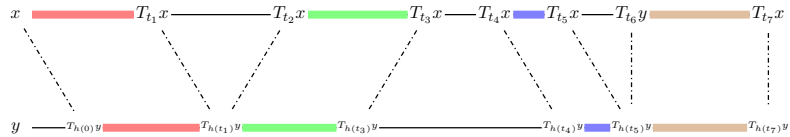
\begin{figure}[H]
 \centering
 \scalebox{0.6}
 {
 \begin{tikzpicture}[scale=5]
	 \tikzstyle{vertex}=[circle,minimum size=20pt,inner sep=0pt]
	 \tikzstyle{selected vertex} = [vertex, fill=red!24]
	 \tikzstyle{edge1} = [draw,line width=5pt,-,red!50]
     \tikzstyle{edge2} = [draw,line width=5pt,-,green!50]
     \tikzstyle{edge3} = [draw,line width=5pt,-,blue!50]
     \tikzstyle{edge4} = [draw,line width=5pt,-,brown!50]

	 \tikzstyle{edge} = [draw,thick,-,black]
     \node[vertex] (v00) at (0.0,0) {$y$};
	 \node[vertex] (v01) at (0.3,0) {\tiny$T_{h(0)}y$};
	 \node[vertex] (v03) at (0.9,0) {\tiny$T_{h(t_1)}y$};
	 \node[vertex] (v04) at (1.5,0) {\tiny$T_{h(t_3)}y$};
	 \node[vertex] (v05) at (2.4,0) {\tiny$T_{h(t_4)}y$};
	 \node[vertex] (v06) at (2.7,0) {\tiny$T_{h(t_5)}y$};
	 \node[vertex] (v07) at (3.3,0) {\tiny$T_{h(t_7)}y$};
     \node[vertex] (v10) at (0,0.5) {$x$};
	 \node[vertex] (v11) at (0.6,0.5) {$T_{t_1}x$};
	 \node[vertex] (v12) at (1.2,0.5) {$T_{t_2}x$};
	 \node[vertex] (v13) at (1.8,0.5) {$T_{t_3}x$};
	 \node[vertex] (v14) at (2.1,0.5) {$T_{t_4}x$};
	 \node[vertex] (v15) at (2.4,0.5) {$T_{t_5}x$};
     \node[vertex] (v16) at (2.7,0.5) {$T_{t_6}y$};
	 \node[vertex] (v17) at (3.3,0.5) {$T_{t_7}x$};
	
	 \draw[edge] (v00)--(v01)--(v03)--(v04)--(v05)--(v06)--(v07);
	 \draw[edge] (v10)--(v11)--(v12)--(v13)--(v14)--(v15)--(v16)--(v17);
     \draw[thick,dash dot] (v10)--(v01);
     \draw[thick,dash dot] (v11)--(v03);
     \draw[thick,dash dot] (v12)--(v03);
     \draw[thick,dash dot] (v13)--(v04);
     \draw[thick,dash dot] (v14)--(v05);
     \draw[thick,dash dot] (v15)--(v06);
     \draw[thick,dash dot] (v16)--(v06);
     \draw[thick,dash dot] (v17)--(v07);
     \draw[edge1] (v01)--(v03);
     \draw[edge1] (v10)--(v11);
     \draw[edge2] (v03)--(v04);
     \draw[edge2] (v12)--(v13);
     \draw[edge3] (v05)--(v06);
     \draw[edge3] (v14)--(v15);
     \draw[edge4] (v06)--(v07);
     \draw[edge4] (v16)--(v17);

 \end{tikzpicture}
 }
 \end{figure}

Based on $(\epsilon,\mathcal{P},R)-$matching, M. Ratner introduced Kakutani invariant. Roughly speaking, the idea to construct the invariant is similar to the construction of the topological entropy. We will count the minimal number of $\epsilon$-balls that needed to cover a large portion of the space $X$, where the $\epsilon$-ball is defined with respect to $f_R$ and $\mathcal{P}$. Then the logarithm of the asymptotic behavior of this minimal number will be our invariant. The precise definition is following:
\begin{definition}[Kakutani invariant, \cite{Ratner4}]
Define $$f_R(x,y,\mathcal{P})=\inf\{\epsilon>0:I_R(x)\text{ and }I_R(y)\text{ are }(\epsilon,\mathcal{P})-\text{matchable}\}.$$ Then for $x\in X$ and $R>1$, let $B_R(x,\epsilon,\mathcal{P})=\{y\in X:f_R(x,y,\mathcal{P})<\epsilon\}$ and  $(\epsilon,R,\mathcal{P})-$cover of $X$ be a family of $B_R(x,\epsilon,\mathcal{P})$ such that their union's measure is larger than $1-\epsilon$. Next let $\alpha_R(\epsilon,\mathcal{P})$ denote $(\epsilon,R,\mathcal{P})-$cover and define $K_R(\epsilon,\mathcal{P})=\inf\operatorname{Card}\{\alpha_R(\epsilon,\mathcal{P})\}$, where infimum is taken over all $(\epsilon,R,\mathcal{P})-$covers of $X$. Let $\mathscr{F}$ be the family of all nondecreasing functions $u:\mathbb{R}^+\to\mathbb{R}^+$ such that $u(t)\to+\infty$ as $t\to+\infty$. Then for $u\in\mathscr{F}$, we define following quantities:
$$\beta(u,\epsilon,\mathcal{P})=\liminf_{R\to\infty}\frac{\log K_R(\epsilon,\mathcal{P})}{u(R)}; \ \ e(u,\mathcal{P})=\limsup_{\epsilon\to0}\beta(u,\epsilon,\mathcal{P}); \ \ e(T_t,u)=\sup_{\mathcal{P}}e(u,\mathcal{P}).$$
\end{definition}

It is shown by M. Ratner [Theorem $3$ in \cite{Ratner3}, Corollary $1$ in \cite{Ratner4}] that if two ergodic flows $(T_t)$ and $(S_t)$ are Kakutani equivalent and $\lim_{t\to\infty}\frac{u(at)}{u(t)}=1$ for all $a>0$, then $e(T_t,u)=e(S_t,u)$. This gives that $e(T_t,u)$ is a \emph{Kakutani invariant}. In the same papers, M. Ratner [\cite{Ratner3}, Theorem $1$, \cite{Ratner4}, Theorem $5$] also proved that this invariant equals to zero for all $u\in\mathscr{U}$ if and only if $(T_t)$ is standard. Moreover, M. Ratner also established the following theorem for Kakutani invariant:

\begin{theorem}[\cite{Ratner3}]\label{thm:generatePartition}
Let $(T_t)$ be an ergodic measure-preserving flow on $(X,\mathscr{B},\mu)$ and let $\mathcal{P}_1\leq\mathcal{P}_2\leq\ldots$ be an increasing sequence of finite measurable partitions of $X$ such that $\vee_{n=1}^{\infty}\mathcal{P}_n$ generates the $\sigma-$algebra $\mathscr{B}$. Then $e(T_t,u)=\sup_me(u,\mathcal{P}_m)$ for all $u\in\mathscr{F}$.
\end{theorem}

The following Lemma is Remark $2.8$ in \cite{KanigowskiVinhageWei2}, we provide a proof here for completeness:
\begin{lemma}\label{compef}
If there exists a set $D\subset X$ with $\mu(D)>\frac{1}{2}$ such that for every $y\in D$ and any $\epsilon\in(0,\frac{1}{2})$, we have
$\mu(B_R(y,\epsilon,\mathcal{P})\cap D)\leq a(R,\epsilon)$,
then $$K_R(\epsilon,\mathcal{P})\geq \frac{1}{a(R,\epsilon)}.$$

Moreover, if for any $\epsilon\in(0,\frac{1}{5})$, there exists a set $D_{\epsilon}$ with $\mu(D_{\epsilon})>1-\epsilon$ such that for every $y\in D_\epsilon$, we have $\mu(B_R(y,\epsilon,\mathcal{P}))\geq b(R,\epsilon)$,
then $$K_R(5\epsilon,\mathcal{P})\leq \frac{1}{b(R,\epsilon)}.$$
\end{lemma}
\begin{proof}
At first recall that by \cite{Ratner4}, we know that  $f_t(\cdot,\cdot,\mathcal{P})$ does not define a metric (triangle inequality fails) but it is close to a metric:
\begin{equation}\label{eq:approMetric}
\text{if $x,y\in
 B_R(z,\epsilon,\mathcal{P})$, then $f_R(x,y,\mathcal{P})<5\epsilon$. }
\end{equation}

If there exists a set $D\subset X$ with $\mu(D)>\frac{1}{2}$ such that for every $y\in D$ and any $\epsilon\in(0,\frac{1}{2})$, we have $\mu(B_R(y,\epsilon,\mathcal{P})\cap D)\leq a(R,\epsilon)$, then the cardinality of $B_R(\cdot,\epsilon,\mathcal{P})$ balls we need to cover $D$ is at least $\frac{1}{a(R,\epsilon)}$ due to direct measure estimate.

If for any $\epsilon\in(0,\frac{1}{5})$, there exists a set $D_{\epsilon}$ with $\mu(D_{\epsilon})>1-\epsilon$ such that for every $y\in D_\epsilon$, we have
\begin{equation}\label{eq:BrLowerBound}
\mu(B_R(y,\epsilon,\mathcal{P}))\geq b(R,\epsilon).
\end{equation}
Then, consider the $5\epsilon-$maximal separated set $\Gamma\subset D_{\epsilon}$ in $X$ with respect to $f_R(\cdot,\cdot,\mathcal{P})$. For any $x_1, x_2\in\Gamma$ and $x_1\neq x_2$, we claim that
\begin{equation}\label{eq:disjointBrball}
B_R(x_1,\epsilon,\mathcal{P})\cap B_R(x_2,\epsilon,\mathcal{P})=\emptyset.
\end{equation}
If this is not the case, then there exists $z\in B_R(x_1,\epsilon,\mathcal{P})\cap B_R(x_2,\epsilon,\mathcal{P})$. By definition of $B_R(\cdot,\epsilon,\mathcal{P})$, we know that $f_R(x_1,z,\mathcal{P})<\epsilon$ and $f_R(x_2,z,\mathcal{P})<\epsilon$, which together with \eqref{eq:approMetric} give us that  $f_R(x_1,x_2,\mathcal{P})<5\epsilon$. As a result, we obtain a contradiction to the choice of $x_1$ and $x_2$ and thus prove the claim.

Notice that \eqref{eq:BrLowerBound} together with \eqref{eq:disjointBrball} guarantee that
$$\operatorname{Card}(\Gamma)\leq \frac{1}{b(R,\epsilon)}.$$
Moreover, the definitions of maximal separated set and minimal covering set give us that: $$\operatorname{Card}(\Gamma)\geq K_R(5\epsilon,\mathcal{P}),$$
which finishes the proof of the lemma.
\end{proof}

\subsection{Special flows}\label{sec:KocherginFlows}
The flows under consideration are a specific type of special flows over $\mathbb{T}$. We recall the basic setting and notation of special flows here for completeness.

Let $\bar{\lambda}$ be the Lebesgue measure on $\mathbb{R}$, $\lambda$ as the Lebesgue measure on $\mathbb{T}$ and $(X,d)$ as a metric space. Suppose $T:(X,\mathscr{B},\mu)\to(X,\mathscr{B},\mu)$ is an automorphism, $f\in L^1(X,\mathscr{B},\mu)$ and $f>0$. Let $X^f=\{(x,s):x\in X,0\leq s\leq f(x)\}$, $\mathscr{B}^f=\mathscr{B}\otimes\mathscr{B}(\mathbb{R})$, $\mu^f=\mu\times\bar{\lambda}$ and $d^f$ be the product metric on $X^f$ induced by the metric $d$ of $X$. The special flow $T^f$ acting on $(X^f,\mathscr{B}^f,\mu^f)$ will move each point in $X^f$ vertically with unit speed and also identity the point $(x,f(x))$ with $(Tx,0)$, i.e. for $x=(x_h,x_v)\in X^f$, we have
\begin{equation}\label{eq:specialflow}
T_t^f(x_h,x_v)=(T^{N(x,t)}x_h,x_v+t-f^{(N(x,t))}(x_h)),
\end{equation}
where $N(x,t)$ is the unique integer such that $f^{(N(x,t))}(x_h)\leq x_v+t<f^{(N(x,t)+1)}(x_h)$ and
$$f^{(n)}(x_h)=\left\{
                 \begin{array}{cc}
                   f(x_h)+\cdots+f(T^{n-1}x_h), & \hbox{if $n>0$;} \\
                   0, & \hbox{if $n=0$;} \\
                   -(f(T^nx_h)+\cdots+f(T^{-1}x_h)), & \hbox{if $n<0$.}
                 \end{array}
               \right.
$$


\subsection{Notation and choice of the partition}\label{sec:Notations}
Throughout this paper we use the following notation. Suppose that $\alpha_1,\alpha_2\in\mathscr{D}$ and $-1<\gamma_1<\gamma_2<0$, let $(\mathscr{T}^{\alpha_1,\gamma_1}_t)$, $(\mathscr{T}^{\alpha_1,\gamma_1}_t)$ be the corresponding special flows and $f_1,f_2$ be the corresponding roof functions satisfying \eqref{eq:powerBehavior} with coefficients $\gamma_1,\gamma_2$, respectively. For $i=1,2$, let $x_1=(x_{1,h},x_{1,v}),\bar{x}_1=(\bar{x}_{1,h},\bar{x}_{1,v})$ belong to $\mathbb{T}^{f_i}$, the metric on $\mathbb{T}^{f_i}$ is of the form $d^{f_i}(x_1,\bar{x}_1)=d_H(x_{1,h},\bar{x}_{1,h})+d_V(x_{1,v},\bar{x}_{1,v})$, where $d_H$ and $d_V$ are Euclidean distances on $\mathbb{T}$ and $\mathbb{R}$ respectively. For $A\subset\mathbb{T}$ and $i=1,2$, let $A^{f_i}:=\{x=(x_h,x_v)\in\mathbb{T}^{f_i},x_h\in A\}$. In order to simplify the notation, we define $\phi_t:=(\mathscr{T}^{\alpha_1,\gamma_1}\times\mathscr{T}^{\alpha_2,\gamma_2})_t$ and $\tilde{\mu}:=\mu^{f_1}\times\mu^{f_2}$.

In order to formulate our proof more efficiently, we introduce following definitions to describe local behavior of an $(\epsilon,\mathcal{P},R)-$matching.
\begin{definition}[Matching balls]\label{def:MatchingBall}
For a fixed $\epsilon>0$, let $x=(x_1,x_2),y=(y_1,y_2)$ be two points in $\mathbb{T}^{f_1}\times\mathbb{T}^{f_2}$ and let $h:A(x,y)\to A'(x,y)$ be an $(\epsilon,\mathcal{P},R)-$matching. For $w\in A(x,y)$ and $R>L>0$, the matching ball around $(w,h(w))$ is defined as
$$B(w,L):=\{r\in A(x,y):w\leq r\leq w+L\}.$$
Moreover, for $t\in A(x,y)$, let
\begin{equation}
\begin{aligned}
x^t&=(x_1^t,x_2^t)=\phi_t(x_1,x_2);\\
y^{h(t)}&=(y_1^{h(t)},y_2^{h(t)})=\phi_{h(t)}(y_1,y_2),
\end{aligned}
\end{equation}
and
\begin{equation}
\begin{aligned}
L_H(t)&:=\max\{d_H(x_1^t,y_1^{h(t)}),d_H(x_2^{t},y_2^{h(t)})\};\\
L(t)&:=\max\{d^{f_1}(x_1^t,y_1^{h(t)}),d^{f_2}(x_2^{t},y_2^{h(t)})\}.
\end{aligned}
\end{equation}
\end{definition}

\begin{definition}[Kakutani Box]\label{def:outBox}
For $x\in\mathbb{T}^{f_1}\times\mathbb{T}^{f_2}$, $\epsilon>0$, $\epsilon_1>0$ and $0<\epsilon_0\leq\frac{|\gamma_2|}{2(1+|\gamma_2|)}$, define
\begin{equation}
\begin{aligned}
\operatorname{Box}^{\operatorname{in}}(x,\epsilon,R):=&\{y\in\mathbb{T}^{f_1}\times\mathbb{T}^{f_2}:0\leq d_H(y_1,x_1)\leq\epsilon R^{-(1+|\gamma_1|+2\epsilon_1)},\\&0\leq d_H(y_2,x_2)\leq\epsilon R^{-(1+|\gamma_2|+2\epsilon_1)},0\leq d_V(y_1,x_1),d_V(y_2,x_2)\leq\epsilon\},\\
\operatorname{Box}^{\operatorname{out}}(x,\epsilon,R):=&\{y\in\mathbb{T}^{f_1}\times\mathbb{T}^{f_2}:0\leq d_H(x_1,y_1),d_H(x_2,y_2)\leq \epsilon R^{-\frac{1}{1-\epsilon_0}}\\
&0\leq d_V(x_1,y_1),d_V(x_2,y_2)\leq\epsilon\},\\
\operatorname{Box}^{\operatorname{out}}_{\epsilon}(x,\epsilon,R):=&\bigcup_{t\in[-\epsilon,\epsilon]}\operatorname{Box}^{\operatorname{out}}(x^t,\epsilon,R).
\end{aligned}
\end{equation}
\end{definition}
\begin{remark}
The motivation of the Definition \ref{def:outBox} is to establish objects that will help us to estimate the measure of Kakutani ball $B_R(x,\epsilon,\mathcal{P})$. More precisely, $\operatorname{Box}^{\operatorname{in}}(x,\epsilon,R)$ will help us estimate the lower bound of Kakutani ball's measure (see Lemma \ref{lm:upperBound1} and \ref{lm:upperBound2}) and $\operatorname{Box}^{\operatorname{out}}_T(x,\epsilon,R)$ will help us estimate the upper bound of the Kakutani ball's measure (see second part of proof of Theorem \ref{thm:Main}).
\end{remark}

Since the estimate of the upper bound of $\tilde{\mu}(\operatorname{Box}^{\operatorname{out}}_{\epsilon}(x,\epsilon,R))$ is not trivial and will be helpful to our proof later, we provide a lemma here to complete this estimate:
\begin{lemma}\label{lm:MovingOurBallMeasure}
For any $y\in\mathbb{T}^{f_1}\times\mathbb{T}^{f_2}$, we have
$$\tilde{\mu}(\operatorname{Box}^{\operatorname{out}}_{\epsilon}(y,\epsilon,R))\leq25\epsilon^4R^{-\frac{2}{1-\epsilon_0}}.$$
\end{lemma}
\begin{proof}
Recall that if $y^t$ does not hit the roof functions for $t\in[-\epsilon,\epsilon]$, then corresponding vertical distance of $\operatorname{Box}^{\operatorname{out}}_{\epsilon}(y^{R_i},\epsilon,R)$ will be $4\epsilon$ in both $\mathbb{T}^{f_1}$ and $\mathbb{T}^{f_2}$, where $2\epsilon$ comes from vertical radius of $\operatorname{Box}^{\operatorname{out}}$ and another $2\epsilon$ comes from $t$; however, if $y_1^t$ hits the roof function for some $t\in[-\epsilon,\epsilon]$, then there are some additional small ``triangle like''(not precise triangles) areas needed to be covered near the roof function $f_1$, which gives the corresponding vertical distance of $\operatorname{Box}^{\operatorname{out}}_{\epsilon}(y^{R_i},\epsilon,R)$ will be $5\epsilon$ in $\mathbb{T}^{f_1}$; thus the worst case for the measure estimate of $\operatorname{Box}^{\operatorname{out}}_{\epsilon}(y^{R_i},\epsilon,R)$ will be $y_1^{t_1}$ and $y_2^{t_2}$ both hit their corresponding roof functions $f_1$ and $f_2$, respectively, for some $t_1,t_2\in[-\epsilon,\epsilon]$, which gives the coefficient $25\epsilon^4$ in the measure estimate of $\operatorname{Box}^{\operatorname{out}}_{\epsilon}(y^{R_i},\epsilon,R)$.
\end{proof}

Due to Theorem \ref{thm:generatePartition}, the estimate of the Kakutani invariant of $\phi_t$ is equivalent to the estimate of the Kakutani invariant of $\phi_t$  with respect to a family of generating partitions of $\mathbb{T}^{f_1}\times\mathbb{T}^{f_2}$. In fact, there is a natural family converging to the point partition of $\mathbb{T}^{f_i}$, $i=1,2$. The method to construct these generating partitions should be understood as following: cut off the cusp part at some height and divide the remaining compact part of $\mathbb{T}^{f_i}$ into small rectangles of small diameters. The detailed steps are as follows: for any positive integer $m$ and $i=1,2$, dividing the set $K_m^i=\{x_i\in\mathbb{T}^{f_i}:f_i(x_{i,h})<2^m\}$ into finitely many atoms of diameter between $\frac{1}{m}$ and $\frac{2}{m}$ with a $C^1$ boundary, then let $\mathcal{P}_m^i$ be the partition consisted of these atoms and a single atom $\mathbb{T}^{f_i}\backslash K_m^i$. We will estimate the Kakutani invariant of $\phi_t$  with respect to $\mathcal{P}_m=\mathcal{P}_m^1\times\mathcal{P}_m^2$ when $m$ is sufficiently large.

\subsection{Denjoy-Koksma Estimates, good sets and preliminary lemmas}
One of the most important tools in our estimate is some quantitative descriptions of the ergodic sums over an irrational rotation for functions with singularities. These description will exactly follow from the Denjoy-Koksma inequality (see \cite{FFK} Lemma $3.1$ for more details). The function $f$ under our consideration satisfies \eqref{eq:powerBehavior} with some $\gamma\in(-1,0)$ and the rotation $Tx=x+\alpha\mod1$ satisfies $\alpha\in\mathscr{D}$. \emph{For simplicity, we will assume that $A_1=B_1=1$ and $\int_{\mathbb{T}}fd\lambda=1$.} Then we have,
\begin{lemma}[Lemma $3.1$ in \cite{FFK}]
For every $z\in\mathbb{T}$, every $M\in\mathbb{Z}$ with $|M|\in[q_{\alpha,s},q_{\alpha,s+1}]$ and $z_{min}^M\doteq\min_{j\in[0,M)}\|z+j\alpha-0\|$, where $\|\cdot\|$ is the Euclidean norm on $\mathbb{T}$, we have
\begin{equation}\label{eq:DenjoyKoksma1}
f(z_{\min}^M)+\frac{1}{3}q_{\alpha,s}\leq f^{(M)}(z)\leq f(z_{\min}^M)+3q_{\alpha,s+1},
\end{equation}
\begin{equation}\label{eq:DenjoyKoksma2}
f'(z_{\min}^M)-8|\gamma|q_{\alpha,s}^{1+|\gamma|}\leq |f'^{(M)}(z)|\leq f'(z_{\min}^M)+8|\gamma|q_{\alpha,s+1}^{1+|\gamma|},
\end{equation}
\begin{equation}\label{eq:DenjoyKoksma3}
f''(z_{\min}^M)\leq f''^{(M)}(z)\leq f''(z_{\min}^M)+8|\gamma(\gamma-1)|q_{\alpha,s+1}^{2+|\gamma|}.
\end{equation}
\end{lemma}

We introduce several sets where the ergodic sums of $f'$ and $f''$ can be controlled. In order to simplify the notation, we will use $(q_n)_{n\geq1}$ and $(q'_n)_{n\geq1}$ as the sequence of denominators of $\alpha_1,\alpha_2$ we introduce in Section \ref{sec:Notations} respectively. The sets $S_n^1$ and $S_n^2$ are defined as following:
\begin{equation}\label{eq:DefS1}
\begin{aligned}
&S_n^1=\left\{x\in\mathbb{T}^{f_1}:\left(\bigcup_{t=-q_n\log q_n}^{q_n\log q_n}\mathscr{T}^{\alpha_1,\gamma_1}_tx\right)\bigcap\left[-\frac{1}{q_n\log^3q_n},\frac{1}{q_n\log^3q_n}\right]^{f_1}=\emptyset\right\},\\
&S_n^2=\left\{x\in\mathbb{T}^{f_2}:\left(\bigcup_{t=-q'_n\log q'_n}^{q'_n\log q'_n}\mathscr{T}^{\alpha_2,\gamma_2}_tx\right)\bigcap\left[-\frac{1}{q'_n\log^3q'_n},\frac{1}{q'_n\log^3q'_n}\right]^{f_2}=\emptyset\right\}.
\end{aligned}
\end{equation}

The asymptotic behaviors \eqref{eq:powerBehavior} of the roof functions imply the measure of $S_n^1$ and $S_n^2$ are large:
\begin{equation}\label{eq:SnControl}
\mu^{f_1}(S_n^1)\geq1-\frac{2^{2+\gamma_1}}{1+\gamma_1}\frac{1}{(\min_{\mathbb{T}}f_1\log^2 q_n)^{1+\gamma_1}},\ \ \mu^{f_2}(S_n^2)\geq1-\frac{2^{2+\gamma_2}}{1+\gamma_2}\frac{1}{(\min_{\mathbb{T}}f_2\log^2 q'_n)^{1+\gamma_2}}.
\end{equation}
The reason we have above inequality for $\mu^{f_1}(S_n^1)$ is following (same lines for $\mu^{f_2}(S_n^2)$): \begin{equation}
\begin{aligned}
\mu^{f_1}(S_n^1)\geq&\int_{\mathbb{T}}f_1(x)dx-2\cdot\int_{0}^{\frac{2q_n\log q_n}{\min_{\mathbb{T}}f_1}\cdot\frac{1}{q_n\log^3q_n}}x^{\gamma_1}dx\\
=&1-\frac{2}{1+\gamma_1}\frac{1}{(\frac{\min_{\mathbb{T}}f_1\log^2 q_n}{2})^{1+\gamma_1}}.\\
\end{aligned}
\end{equation}

For $n_1\in\mathbb{N}$, we define:
\begin{equation}\label{eq:intersectionSets}
S^1(n_1):=\bigcap_{n\geq n_1}S_n^1,\ \ S^2(n_1):=\bigcap_{n\geq n_1}S_n^2.
\end{equation}

Indeed, we have the following lemma to estimate the measure of $S^i(n_1)$, which directly follows from the fact that $\{q_n\}$ grow at least exponentially:
\begin{lemma}\label{lm:Snmeaure}
For any $\delta>0$, there exists $n_1=n_1(\delta)\in\mathbb{N}$ such that $\mu^{f_i}(S^{i}(n_1))>1-\delta^3$ for $i=1,2$.
\end{lemma}

In the above setting, we have the following two lemmas to deal with the differentiability of ergodic sums of $f_1,f_2$ and the bounds of the ergodic sums of $f'_1,f'_2$ respectively:
\begin{lemma}[Lemma $2.7$ in \cite{KanigowskiWei1}]\label{lm:differentialF}
There exists a constant $d_i=d_i(f_i,\alpha_i)>0$ such that for $z\in S^i(n_1)$ satisfying $d_H(z,z')\leq d_i$ and $z_h<z'_h$, we have
$$0\notin[z_h+w\alpha_i,z'_h+w\alpha_i],\forall w\in(0,|\frac{d_H(z,z')^{-1}}{\log^{12}d_H(z,z')}|),$$
for $i=1,2$.
\end{lemma}
\begin{remark}
The differentiability follows from this lemma as the singularity will stay away from the $[z_h+w\alpha_i,z'_h+w\alpha_i]$ up to certain iterations.
\end{remark}

\begin{lemma}[Lemma $2.9$ in \cite{KanigowskiWei1}]\label{lm:derivativeControl}
For every $\epsilon_1>0$, there exist $n'\in\mathbb{N}$ and $\delta_0>0$ such that for all $\delta_0>\delta>0$ and $n_1\geq n'$, there exist sets $W^i(\delta)\subset\mathbb{T}^{f_i}$ with $\mu^{f_i}(W^i(\delta))\geq1-\delta^{10}$ and $S^i(n_1)$ with $\mu^{f_i}(S^i(n_1))\geq1-\delta^{3}$ such that  for all $x_1\in S^1(n_1)\cap W^1(\delta)$, $x_2\in S^2(n_1)\cap W^2(\delta)$ and $T\geq n'$, there exists a set $G_T\subset[0,T]$ with $\bar{\lambda}(G_T)\geq T(1-4\log^{-3}T)$ such that for every $t\in G_T$,
\begin{equation}\label{eq:derivativeControl1}
t^{1+|\gamma_1|-\epsilon_1}\leq|f_1'^{(N(x,t))}(\theta_h)|\leq t^{1+|\gamma_1|+\epsilon_1}\text{ if } d_H(\theta,x_{1})\leq(T\log^{2P_1}T)^{-1},
\end{equation}
\begin{equation}\label{eq:derivativeControl2}
t^{1+|\gamma_2|-\epsilon_1}\leq|f_2'^{(M(x,t))}(\xi_h)|\leq t^{1+|\gamma_2|+\epsilon_1}\text{ if } d_H(\xi,x_{2})\leq(T\log^{2P_2}T)^{-1},
\end{equation}
where $N(x,t)$ and $M(x,t)$ are defined in \eqref{eq:specialflow} for respectively $\mathscr{T}^{\alpha_1,\gamma_1}_t$ and $\mathscr{T}^{\alpha_2,\gamma_2}_t$, $P_1=100|\gamma_1|^{-1}$ and $P_2=100|\gamma_2|^{-1}$.
\end{lemma}
\begin{remark}\label{remark}
In fact, based on the proof of Lemma $2.9$ in \cite{KanigowskiWei1}, the upper bound of the derivatives of Birkhoff sums holds for any $t\in[0,T]$.
\end{remark}

By using Lemma \ref{lm:derivativeControl} to control the Birkhoff sums' derivatives, we have the following lemma to describe the behavior of Kakutani box under the action of $\phi_t$. More precisely, this lemma shows that if two points $x=(x_1,x_2),y=(y_1,y_2)$ are close, they will stay in same atom of the partition by adding a small vertical perturbation to one of the two points. This phenomenon is natural in general case but due to the distance and partition in our setting, there is a possibility that two points may be close at time $t$ but $x^t_{1,v}$ is close to the roof function and $x^t_{2,v}$ is close to the base and thus cannot in the same atom of the partition. From now on, we define \begin{equation}\label{def:ebound}
\epsilon_{bound}=\min\{\frac{1}{100},\frac{1}{100}\min_{x\in\mathbb{T}}f_1,\frac{1}{100}\min_{x\in\mathbb{T}}f_2\}.
\end{equation}

Before we establish the lemmas to control the Birkhoff sums' derivatives, we state the following useful lemma to simplify our proof:
\begin{lemma}\label{lm:hittingTimeEstimates}
For every $\epsilon,\epsilon_1\in(0,\epsilon_{bound})$, there exists $N_{0,\epsilon}\in\mathbb{R}$ and a set $\bar{E}_{\epsilon}\subset\mathbb{T}^{f_1}\times\mathbb{T}^{f_2}$ with $\tilde{\mu}(\bar{E}_{\epsilon})>1-\epsilon$ such that for every $R>N_{0,\epsilon}$, $x\in\bar{E}_{\epsilon}$, $y\in\mathbb{T}^{f_1}\times\mathbb{T}^{f_2}$ and $a_1,a_2,a_3>0$ with $$a_1-(1+|\gamma_1|+\epsilon_1)(1+\frac{1}{50(1+|\gamma_1|)}\epsilon_1)a_3>0, \ \ a_2-(1+|\gamma_2|+\epsilon_1)(1+\frac{1}{50(1+|\gamma_1|)}\epsilon_1)a_3>0,$$ if $d_H(x_1,y_1)\in[0, \epsilon^2R^{-a_1}]$, $d_H(x_2,y_2)\in[0, \epsilon^2R^{-a_2}]$ and $d_V(x_1,y_1),d_V(x_2,y_2)\in[0,\epsilon^2]$, then we have:
\begin{enumerate}[(1)]
\item For any $L\in[0,R^{a_3}]$, we have
$$N_1(x_1,L)-2\leq N_1(y_1,L)\leq N_1(x_1,L)+2,$$
$$N_2(x_2,L)-2\leq N_2(y_2,L)\leq N_2(x_2,L)+2.$$
\item There exists a set $D_R(x)\subset[0,R^{a_3}]$ with $\bar{\lambda}(D_R(x))\geq(1-\frac{1}{10}\epsilon)R^{a_3}$ such that for every $L\in D_R(x)$, we have $$N_1(y_1,L)=N_1(x_1,L),\text{ and }N_2(y_2,L)=N_2(x_2,L).$$
\end{enumerate}
\end{lemma}
\begin{proof}
Pick $\delta\in(0,\min\{\delta_0,\epsilon\})$, $N_{0,\epsilon}=\max\{n',1,\epsilon^{-5}\epsilon_{bound}\}$ and define $\bar{E}_{\epsilon}=(S^1(n')\cap W^1(\delta))\times(S^2(n')\cap W^2(\delta))$, where $\delta_0$, $n'$, $S^i(n')$ and $W^i(\delta)$ are from Lemma \ref{lm:derivativeControl} for $i=1,2$. By construction, we have $\tilde{\mu}(\bar{E}_{\epsilon})>1-\epsilon$.

If $d_H(x_i,y_i)\in[0,R^{-a_i}]$ for $i=1,2$ and also recall that $x\in\bar{E}_{\epsilon}$, then we can use Lemma \ref{lm:differentialF} and thus for any $L\in [0,R]$, we have
$$|f_i^{(N_i(x_i,L))}(y_{i,h})-f_i^{(N_i(x_i,L))}(x_{i,h})|=|f_i'^{(N_i(x_i,L))}(\theta_i)||x_{i,h}-y_{i,h}| \text{ for }i=1,2,$$
where $\theta_i$ is in the shorter one of the arcs given by $[x_{i,h},y_{i,h}]$ for $i=1,2$ and $N_1(x_1,L)$ and $N_2(x_2,L)$ are defined in \eqref{eq:specialflow} for respectively $\mathscr{T}^{\alpha_1,\gamma_1}$ and $\mathscr{T}^{\alpha_2,\gamma_2}$. By Lemma \ref{lm:derivativeControl}, Remark \ref{remark}, $L\in[0,R^{a_3}]$, $|\gamma_2|<|\gamma_1|$ and choice of $a_1$, $a_2$, $a_3$, we have for $i=1,2$:
\begin{equation}\label{eq:OutboxDer1}
\begin{aligned}
|f_i^{(N_i(x_i,L))}(y_{i,h})-f_i^{(N_i(x_i,L))}(x_{i,h})|\leq \epsilon^2 L^{1+|\gamma_i|+\epsilon_1}R^{-a_i}\leq\epsilon^2 R^{a_3(1+|\gamma_i|+\epsilon_1)-a_i}\leq\epsilon^2.
\end{aligned}
\end{equation}

Recall $x\in\bar{E}_{\epsilon}$, definition of $S^1(n')$ (see \eqref{eq:intersectionSets}), definition of $S^2(n')$ (see \eqref{eq:intersectionSets}), definition of $S^1_n$ (see \eqref{eq:DefS1}) and Diophantine condition of $q_n$, we know that $$x_{1,h}^L,x_{2,h}^L\notin[-R^{-(1+\frac{1}{100(1+|\gamma_1|)}\epsilon_1)a_3},R^{-(1+\frac{1}{100(1+|\gamma_1|)}\epsilon_1)a_3}],$$
then we apply \eqref{eq:DefS1} and \eqref{eq:intersectionSets} again, for any $\zeta_i\in[x_{i,h}^L-\epsilon^2R^{-a_i},x_{i,h}^L+\epsilon^2R^{-a_i}]$ where $i=1,2$, we obtain that $$\zeta_1,\zeta_2\notin[-R^{-(1+\frac{1}{50(1+|\gamma_1|)}\epsilon_1)a_3},R^{-(1+\frac{1}{50(1+|\gamma_1|)}\epsilon_1)a_3}],$$ and thus,
\begin{equation}\label{eq:roofDiff}
\begin{aligned}
\epsilon^2|f_i'(\zeta_i)|R^{-a_i}\leq \epsilon^2R^{(1+\frac{1}{50(1+|\gamma_1|)}\epsilon_1)(1+|\gamma_i|)a_3-a_i}\leq\epsilon^2\text{ for }i=1,2.
\end{aligned}
\end{equation}

In fact, \eqref{eq:OutboxDer1} and \eqref{eq:roofDiff} gives the difference of Birkhoff sums at $N_i(x_i,L)+1$ for $i=1,2$:
\begin{equation}\label{eq:SpecialN+1}
\begin{aligned}
|f_i^{(N_i(x_i,L)+1)}&(y_{i,h})-f_i^{(N_i(x_i,L)+1)}(x_{i,h})|\\&=|f_i^{(N_i(x_i,L))}(y_{i,h})-f_i^{(N_i(x_i,L))}(x_{i,h})+f_i(R_{\alpha_i}^{N_i(x_i,L)}y_{i,h})-f_i(x_{i,h}^L)|\\
&\leq|f_i^{(N_i(x_i,L))}(y_{i,h})-f_i^{(N_i(x_i,L))}(x_{i,h})|+|f_i(R_{\alpha_i}^{N_i(x_i,L)}y_{i,h})-f_i(x_{i,h}^L)|\\
&\leq\epsilon^2+|f_i'(\zeta_i)||R_{\alpha_i}^{N_i(x_i,L)}y_{i,h}-x_{i,h}^L|\\
&\leq\epsilon^2+|f_i'(\zeta_i)|\epsilon^2R^{-a_i}\leq2\epsilon^2,
\end{aligned}
\end{equation}
where $\zeta_i\in[x_{i,h}^L-\epsilon^2R^{-a_i},x_{i,h}^L+\epsilon^2R^{-a_i}]$ and line $3$ to line $4$ follows from that $$|R_{\alpha_i}^{N_i(x_i,L)}y_{i,h}-x_{i,h}^L|=|y_{i,h}-x_{i,h}|\leq\epsilon^2R^{-a_i},$$ and $x_{i,h}^L\notin[-R^{-(1+\frac{1}{100(1+|\gamma_1|)}\epsilon_1)a_3},R^{-(1+\frac{1}{100(1+|\gamma_1|)}\epsilon_1)a_3}]$.

Recall the definition of $N_1(x_1,L)$, we have
\begin{equation}\label{eq:specialcontrol}
f_1^{(N_1(x_1,L))}(x_{1,h})\leq L<f_1^{(N_1(x_1,L)+1)}(x_{1,h}).
\end{equation}
Then \eqref{eq:OutboxDer1}, \eqref{eq:SpecialN+1}, \eqref{eq:specialcontrol} and choice of $\epsilon$ imply that
\begin{equation}\label{eq:SecondNECount}
\begin{aligned}
f_1^{(N_1(x_1,L)-1)}(y_{1,h})&<f_1^{(N_1(x_1,L))}(y_{1,h})-\epsilon^2\leq L\\&<f_1^{(N_1(x_1,L)+1)}(y_{1,h})+2\epsilon^2<f_1^{(N_1(x_1,L)+2)}(y_{1,h}),
\end{aligned}
\end{equation}
Since the definition of $N_1(y_1,L)$ implies
\begin{equation}\label{eq:NspecialDef}
f_1^{(N_1(y_1,L))}(y_{1,h})\leq L<f_1^{(N_1(y_1,L)+1)}(y_{1,h})
\end{equation}
Thus \eqref{eq:SecondNECount} and \eqref{eq:NspecialDef} imply
$$N_1(x_1,L)-2\leq N_1(y_1,L)\leq N_1(x_1,L)+2.$$
By similar arguments, we also obtain that
$$N_2(x_2,L)-2\leq N_2(y_2,L)\leq N_2(x_2,L)+2,$$
which finishes the proof of first part of the Lemma.

For any $\epsilon\in(0,\epsilon_{bound})$, define sets $B_{\epsilon^2}^1$ and $B_{\epsilon^2}^2$ as following:
\begin{equation}\label{eq:setB}
\begin{aligned}
B_{\epsilon^2}^1=\{x\in\mathbb{T}^{f_1}\times\mathbb{T}^{f_2}: & \ \ 0\leq x_{1,v}\leq\epsilon^2\text{ or }0\leq x_{2,v}\leq\epsilon^2\},\\
B_{\epsilon^2}^2=\{x\in\mathbb{T}^{f_1}\times\mathbb{T}^{f_2}:& \ \ f_1(x_{1,h})-2\epsilon^2\leq x_{1,v}\leq f_1(x_{1,h})\\
\text{ or }&f_2(x_{2,h})-2\epsilon^2\leq x_{2,v}\leq f_2(x_{2,h})\}.
\end{aligned}
\end{equation}
Notice that $\tilde{\mu}(B_{\epsilon^2}^1\cup B_{\epsilon^2}^2)\leq6\epsilon^2$.

Then for any $x\in \bar{E}_{\epsilon}$, we define
\begin{equation}\label{eq:ChoiceOfX}
D_R(x)=\{t\in[0,R^{a_3}]:x^t\notin(B_{\epsilon^2}^1\cup B_{\epsilon^2}^2)\}.
\end{equation}
It is also worth to point out that if $R^{a_3}>\epsilon^{-5}\epsilon_{bound}$, by the definition of special flow and sets $B_{\epsilon^2}^1$, $B_{\epsilon^2}^2$, we have $\bar{\lambda}(D_R(x))>(1-\frac{1}{10}\epsilon)R^{a_3}$.

Thus for every $L\in D_R(x)$, we have the following estimate of $N_1(x_1,L)$:
\begin{equation}\label{eq:specialEcontrol}
f_1^{(N_1(x_1,L))}(x_{1,h})+\epsilon^2\leq L<f_1^{(N_1(x_1,L)+1)}(x_{1,h})-2\epsilon^2.
\end{equation}
Then \eqref{eq:OutboxDer1}, \eqref{eq:SpecialN+1} and \eqref{eq:specialEcontrol} imply that
\begin{equation}\label{eq:SecondCount}
f_1^{(N_1(x_1,L))}(y_{1,h})\leq L<f_1^{(N_1(x_1,L)+1)}(y_{1,h}),
\end{equation}
which gives $N_1(y_1,L)=N_1(x_1,L)$ by \eqref{eq:specialflow}. By the similar argument, we also obtain that $N_2(y_2,L)=N_2(x_2,L)$ and thus finish the proof of second part of the Lemma.
\end{proof}

Now we have enough preparation to state the control lemma of Birkhoff sums' derivative:

\begin{lemma}\label{lem:linearControl}
For every $\epsilon\in(0,\epsilon_{bound})$, there exists $N_{0,\epsilon}$ such that for every $R>N_{0,\epsilon}$, there exists a set $\bar{E}_{\epsilon}\subset\mathbb{T}^{f_1}\times\mathbb{T}^{f_2}$ with $\tilde{\mu}(\bar{E}_{\epsilon})>1-\epsilon$, such that for every $x\in \bar{E}_{\epsilon}$ and $y\in\mathbb{T}^{f_1}\times\mathbb{T}^{f_2}$, if $x\in\operatorname{Box}^{\operatorname{out}}(y,\epsilon^2,R)$, then for every $L\in [0,R^{\frac{1}{(1-\epsilon_0)(1+|\gamma_1|+2\epsilon_1)}}]$, we have
$$x^{L'}\in\operatorname{Box}^{\operatorname{out}}_{\epsilon}(y^L,\epsilon,R),$$
where $|L'-L|<6\epsilon^2$.
\end{lemma}
\begin{proof}
Let $\epsilon_1\in(0,\epsilon_{bound}]$, then we apply Lemma \ref{lm:hittingTimeEstimates} with $\epsilon=\epsilon$, $\epsilon_1=\epsilon_1$, $a_1=a_2=\frac{1}{1-\epsilon_0}$ and $a_3=\frac{1}{(1-\epsilon_0)(1+|\gamma_1|+2\epsilon_1)}$, where $\epsilon_0\in(0,\frac{|\gamma_2|}{2(1+|\gamma_1|)}]$. By the second part of Lemma \ref{lm:hittingTimeEstimates}, we obtain that for any $x\in\bar{E}_{\epsilon}$, $y\in\mathbb{T}^{f_1}\times\mathbb{T}^{f_2}$ and $x\in\operatorname{Box}^{\operatorname{out}}(y,\epsilon^2,R)$ (see Definition \ref{def:outBox}), there exists $D_R(x)\subset[0,R^{\frac{1}{(1-\epsilon_0)(1+|\gamma_1|+2\epsilon_1)}}]$ with $\bar{\lambda}(D_R(x))\geq(1-\frac{1}{10}\epsilon)R^{\frac{1}{(1-\epsilon_0)(1+|\gamma_1|+2\epsilon_1)}}$ such that for any $L\in D_R(x)$, we have
$$N_1(y_1,L)=N_1(x_1,L),\text{ and }N_2(y_2,L)=N_2(x_2,L).$$





By the special flow representation, we have $x_i^{L}=(x_{i,h}+N_i(x_i,L)\alpha_i,L-f_i^{(N(x_i,L))}(x_{i,h}))$ and $y_i^{L}=(y_{i,h}+N_i(y_i,L)\alpha_i,L-f_i^{(N(y_i,L))}(y_{i,h}))$ for $i=1,2$. Then $N_i(x_i,L)=N_i(y_i,L)$ implies $d_H(x_i^L,y_i^L)=d_H(x_i,y_i)=\epsilon^2R^{-\frac{1}{1-\epsilon_0}}<\epsilon R^{-\frac{1}{1-\epsilon_0}}$; $N_i(x_i,L)=N_i(y_i,L)$ and \eqref{eq:OutboxDer1} implies that $d_V(x_i^L,y_i^L)\leq d_V(x_i,y_i)+\epsilon^2<4\epsilon^2$. This gives that $x^L\in\operatorname{Box}^{\operatorname{out}}(y^L,4\epsilon^2,R)$ when $L\in D_R(x)$.

Now for any $L\in[0,R^{\frac{1}{(1-\epsilon_0)(1+|\gamma_1|+2\epsilon_1)}}]$, considering the intervals $U_L=[L,L+10\epsilon^2]$. By the definitions of $B_{\epsilon^2}^1$ and $B_{\epsilon^2}^2$, there exist $u_L\in U_L\cap D_R(x)$. More precisely, $u_L\in U_L\cap D_R(x)$ follows from for any $T>0$:
\begin{equation}
\begin{aligned}
\bar{\lambda}(\{t:T\leq t\leq T+10\epsilon^2,y^t\in(B_{\epsilon^2}^1\cup B_{\epsilon^2}^2)\})\leq 6\epsilon^2,
\end{aligned}
\end{equation}
and the reason we have above inequality is $10\epsilon^2<\epsilon<\min_{i=1,2}\{\min_{\mathbb{T}}f_i\}$.

Notice that $x^{u_L}\in\operatorname{Box}^{\operatorname{out}}(y^{u_L},4\epsilon^2,R)$, $u_L\in U_L$ and $10\epsilon^2<\epsilon<\min_{i=1,2}\{\min_{\mathbb{T}}f_i\}$, we obtain that $x^{u_L}\in \operatorname{Box}^{\operatorname{out}}_{\epsilon}(y^L,\epsilon,R)$ as $\operatorname{Box}^{\operatorname{out}}(y^{u_L},4\epsilon^2,R)\subset\operatorname{Box}^{\operatorname{out}}_{\epsilon}(y^L,\epsilon,R)$. This finishes the proof.
\end{proof}

\section{Proof of Theorem \ref{thm:Main}}\label{sec:mainThm}
In this section we will prove Theorem \ref{thm:Main}. The proof is quite complicated and thus we divided it into several sections.

\paragraph{Outline of the proof:} The proof first shows that $e(\phi_t,\log)\leq 1+|\gamma_1|+|\gamma_2|$ and then $e(\phi_t,\log)\geq \frac{2+4|\gamma_2|+2|\gamma_1\gamma_2|}{(2+|\gamma_2|)(1+|\gamma_1|)}$. The proof of the upper bound follows from the relation between $B_R(y,\epsilon,\mathcal{P}_m)$ and $\operatorname{Box}^{\operatorname{in}}(y,\epsilon^4,\mathcal{P}_m)$. More precisely, $B_R(y,\epsilon,\mathcal{P}_m)$ approximately contains $\epsilon R$ different copies of $\operatorname{Box}^{\operatorname{in}}(y_i,\epsilon^4,\mathcal{P}_m)$ for some $y_i$, since the $f_R$ metric has $\epsilon R$ flexibility along the orbit direction. Once we estimate the lower bound of $\tilde{\mu}(B_R(x,\epsilon,\mathcal{P}_m))$ from the above relation, we will obtain the upper bound by applying Lemma \ref{compef}. As for the estimates of the lower bound, we connect $B_R(y,\epsilon,\mathcal{P}_m)$ with $\operatorname{Box}^{\operatorname{out}}(y,\epsilon,R)$. More precisely, we use Theorem \ref{lm:mainLemma} to show that roughly $B_R(y,\epsilon,\mathcal{P}_m)$ can be covered by $R$ copies of $\operatorname{Box}^{\operatorname{out}}(y_i,\epsilon,R)$ for some $y_i$. Once we obtain the upper bound of $\tilde{\mu}(B_R(y,\epsilon,\mathcal{P}_m))$, we will obtain the lower bound by applying Lemma \ref{compef}.

The following theorem is a crucial step in estimating the lower bounds of Kakutani invariant as it shows that two points in one Kakutani ball implies that the orbits of these two points are polynomial close.

\begin{lm:mainLemma}
For every $\delta>0$ there exists a set $D=D_{\delta}\subset\mathbb{T}^{f_1}\times\mathbb{T}^{f_2}$ with $\tilde{\mu}(D)>1-\delta$ and $m_{\delta},R_{\delta}\in\mathbb{R}$ such that for every $x=(x_1,x_2),y=(y_1,y_2)\in D$, $m\geq m_{\delta}$, $R\geq R_{\delta}$ and $x,y$ are $(\frac{1}{100},\mathcal{P}_m,R)-$matchable, there exists $t_0\in A(x,y)$\footnote{$A(x,y)$ is defined in Definition \ref{def:ematchable}.} such that $x^{t_0}\in (S^1(n')\cap W^1(\delta))\times(S^2(n')\cap W^2(\delta))$ and
\begin{equation}
L_H(t_0)\leq R^{-\frac{1}{1-\epsilon_0}},
\end{equation}
where $\epsilon_0=\frac{|\gamma_2|}{2(1+|\gamma_2|)}-4\epsilon_2\frac{2+|\gamma_2|}{(|\gamma_1|-|\gamma_2|)(1+|\gamma_2|)}$ and $\epsilon_2\in(0,\frac{|\gamma_2|(|\gamma_1|-|\gamma_2|)}{16(2+|\gamma_2|)})$ is a fixed number.
\end{lm:mainLemma}

\begin{remark}
The particular choices of $\epsilon_0$ and $\epsilon_2$ in Theorem \ref{lm:mainLemma} are because we want to obtain the optimal lower bound of Kakutani invariant. The restrictions of the choices of $\epsilon_0$ and $\epsilon_2$ are the opposite inequalities of \eqref{eq:restrictionChoice1} and \eqref{eq:restrictionChoice2}.
\end{remark}

Before we prove Theorem \ref{lm:mainLemma}, let us give a conditional proof of Theorem \ref{thm:Main} first.
\begin{proof}[Proof of Theorem \ref{thm:Main}]
We will prove the upper bound from a specific matching construction and then prove the lower bound by Theorem \ref{lm:mainLemma}.

\noindent\textbf{Upper bound estimate:}

Fix $m\in\mathbb{N}$ and let $0<\epsilon<\min\{\frac{1}{100},\frac{1}{100}\min_{i=1,2}\{\min_{x\in\mathbb{T}}f_i\},\frac{1}{m^3}\}$. Suppose $V_{m}^{\epsilon^2}$ is the $\epsilon^2$ neighborhood of the boundary of $\mathcal{P}_m$. As these boundaries are $C^1$ it implies that $\tilde{\mu}(V_m^{\epsilon^2})=O(m^2\epsilon^4)$. Then we define sets $B_{\epsilon^2}^1$ and $B_{\epsilon^2}^2$ as in \eqref{eq:setB}.

By applying the ergodic theorem to $\phi_t$ and the set $V^{\epsilon^2}_m\cup B_{\epsilon^2}^1\cup B_{\epsilon^2}^2$, we obtain there exists a set $E'_{\epsilon}$ with $\tilde{\mu}(E'_{\epsilon})>1-m\epsilon^2$ and a number $N_{\epsilon}>0$ such that for every $R\geq N_{\epsilon}$ and every $y\in E'_{\epsilon}$, we have
\begin{equation}
\bar{\lambda}(\{t\in[0,R]:y^t\in V^{\epsilon^2}_m\cup B_{\epsilon^2}^1\cup B_{\epsilon^2}^2\})\leq\frac{\epsilon R}{2}.
\end{equation}

For every $\epsilon_1>0$, let $n'$, $\delta$, $S^i(n')$ and $W^i(\delta)$ be defined as in Lemma \ref{lm:derivativeControl} for $i=1,2$. Then define set $E_{\epsilon}=E'_{\epsilon}\cap\bar{E}_{\epsilon}$.


Also notice that for every $\gamma_1\in(-1,0)$, there exists $C_{\gamma_1}>0$ such that $t^{|\gamma_1|/2}>\log^{11}t$ when $t\geq C_{\gamma_1}$. Finally let $C_0>0$ be the constant such that $\frac{t}{\log^6t}$ is increasing for $t\geq C_0$. Then define $C_{\alpha_1,\gamma_1}=(\max\{N_{\alpha_1},C_{\gamma_1},C_0,1\})^{\frac{1}{1-|\gamma_1|}}$, where $N_{\alpha_1}$ is defined in \eqref{eq:Ncontrol} by letting $\alpha=\alpha_1$.

Our estimate of the upper bound will follow from the following two lemmas:
\begin{lemma}\label{lm:upperBound1}
For every $y\in E_{\epsilon}$ and every $R\geq \max\{N_{\epsilon},C_{\alpha_1,\gamma_1}\}$ and all $x\in\mathbb{T}^{f_1}\times\mathbb{T}^{f_2}$, if there exists $p\in[0,\epsilon^3R]$ such that $x^p\in \operatorname{Box}^{\operatorname{in}}(y,\epsilon^4,R)$, then $x\in B_R(y,\epsilon,\mathcal{P}_m)$.
\end{lemma}

\begin{lemma}\label{lm:upperBound2}
For every $y\in E_{\epsilon}$, every $R\geq \max\{N_{\epsilon},C_{\alpha_1,\gamma_1}\}$ and every $p,q\in[0,\epsilon^3R]$ satisfying $|p-q|\geq C_{\alpha_1,\gamma_1}$, we have
$$\phi_{-p}(\operatorname{Box}^{\operatorname{in}}(y,\epsilon^4,R))\bigcap \phi_{-q}(\operatorname{Box}^{\operatorname{in}}(y,\epsilon^4,R))=\emptyset.$$
\end{lemma}

Before we prove these lemmas, we show at first how these lemmas implies the upper bound.

Take $y\in E_{\epsilon}$, by Lemma \ref{lm:upperBound1} it follows that
$$\bigcup_{p\in[0,\epsilon^3R]}\phi_{-p}(\operatorname{Box}^{\operatorname{in}}(y,\epsilon^4,R))\subset B_R(y,\epsilon,\mathcal{P}_m).$$

Thus by Lemma \ref{lm:upperBound2} and $\epsilon<\frac{1}{100}\min_{i=1,2}\{\min_{\mathbb{T}}f_i\}$, we obtain:
\begin{equation}
\begin{aligned}
\tilde{\mu}(B_R(y,\epsilon,\mathcal{P}_m))\geq&\tilde{\mu}\left(\bigcup_{p\in[0,\epsilon^3R]}\phi_{-p}(\operatorname{Box}^{\operatorname{in}}(y,\epsilon^4,R))\right)\\
\overset{\text{Lemma \ref{lm:upperBound2}}}{\geq}& C_{\alpha_1,\gamma_1}^{-1}\epsilon^4R\tilde{\mu}(\operatorname{Box}^{\operatorname{in}}(y,\epsilon^4,R))\\
\overset{\text{Definition \ref{def:outBox}}}{\geq} &4C_{\alpha_1,\gamma_1}^{-1}\epsilon^4R\epsilon^{16}R^{-(1+|\gamma_1|+2\epsilon_1)}R^{-(1+|\gamma_2|+2\epsilon_1)}\\
=&4C_{\alpha_1,\gamma_1}^{-1}\epsilon^{20}R^{-(1+|\gamma_1|+|\gamma_2|+4\epsilon_1)}.\\
\end{aligned}
\end{equation}

Since this holds for every $y\in E_{\epsilon}$ and $\tilde{\mu}(E_{\epsilon})>1-\epsilon$, it follows from Lemma \ref{compef}, there exists some $C(\epsilon,\alpha_1,\gamma_1)>0$ depending only on $\epsilon$, $\alpha_1$ and $\gamma_1$ such that
\begin{equation}
K_R(5\epsilon,\mathcal{P}_m)\leq C(\epsilon,\alpha_1,\gamma_1)R^{1+|\gamma_1|+|\gamma_2|+4\epsilon_1}.
\end{equation}
As a result, we obtain that $\beta(\log,5\epsilon,\mathcal{P}_m)\leq 1+|\gamma_1|+|\gamma_2|+4\epsilon_1$ and thus Theorem \ref{thm:generatePartition} and the arbitrariness of $\epsilon_1$ imply that $$e(\phi_t,\log)\leq 1+|\gamma_1|+|\gamma_2|.$$
\begin{proof}[Proof of Lemma \ref{lm:upperBound1}]
Suppose $y\in E_{\epsilon}$ and $D_R'(y)=\{t\in[0,R]:y^t\notin V_{\epsilon^2}\}$. Let $\epsilon_1\in(0,\epsilon_{bound}]$, then we apply Lemma \ref{lm:hittingTimeEstimates} with $\epsilon=\epsilon^2$, $\epsilon_1=\epsilon_1$, $a_1=1+|\gamma_1|+2\epsilon_1$, $a_2=1+|\gamma_2|+2\epsilon_1$ and $a_3=1$. By the second part of Lemma \ref{lm:hittingTimeEstimates}, we obtain that for any $y\in E_{\epsilon}\subset\bar{E}_{\epsilon}$, $x^p\in\mathbb{T}^{f_1}\times\mathbb{T}^{f_2}$ and $x^p\in\operatorname{Box}^{\operatorname{in}}(y,\epsilon^4,R)$, there exists $D_R(y)\subset[0,R]$ with $\bar{\lambda}(D_R(y))\geq(1-\frac{1}{10}\epsilon^2)R$ such that for any $t\in D_R(y)$, we have
$$N_1(x_1^p,t)=N_1(y_1,t),\text{ and }N_2(x_2^p,t)=N_2(y_2,t).$$

By the special flow representation, we have $x_i^{p+t}=(x_{i,h}^p+N_i(x_i^p,t)\alpha_i,t-f_i^{(N(x_i^p,t))}(x_{i,h}^p))$ and $y_i^{t}=(y_{i,h}+N_i(y_i,t)\alpha_i,t-f_i^{(N(y_i,t))}(y_{i,h}))$ for $i=1,2$. Then $N_i(x_i^p,t)=N_i(y_i,t)$ implies $d_H(x_i^{p+t},y_i^t)=d_H(x_i^p,y_i)=\epsilon^4R^{-(1+|\gamma_i|+2\epsilon_1)}<\epsilon^3$; $N_i(x_i^p,t)=N_i(y_i,t)$ and \eqref{eq:OutboxDer1} implies that $d_V(x_i^{p+t},y_i^t)\leq d_V(x_i^p,y_i)+\epsilon^4<\epsilon^3$. Recall that for $t\in D_R'(y)\cap D_R(y)$, $y^t\notin V_{\epsilon^2}$ and $\epsilon<\frac{1}{m^3}$, thus above implies $y^t$ and $x^{p+t}$ will stay in same atom of partition $\mathcal{P}_m$. Notice that $\bar{\lambda}(D_R'(y)\cap D_R(y))\geq(1-\frac{2}{3}\epsilon)R$ and $p\in[0,\epsilon^3R]$, this gives $x\in B_R(y,\epsilon,\mathcal{P}_m)$ and thus finishes the proof.
\end{proof}

\begin{proof}[Proof of Lemma \ref{lm:upperBound2}]
We prove this lemma by contradiction. If the statement of the lemma is not hold, then there exist $p_0,q_0\in[0,\epsilon^3R]$ and $x\in\mathbb{T}^{f_1}\times\mathbb{T}^{f_2}$ such that $$x\in\operatorname{Box}^{\operatorname{in}}(y,\epsilon^4,R)\bigcap \phi_{p_0-q_0}(\operatorname{Box}^{\operatorname{in}}(y,\epsilon^4,R)).$$ Without loss of generality, we assume that $p_0> q_0$\footnote{Otherwise we will consider $x\in(\mathscr{T}^{\alpha_1,\gamma_1}\times\mathscr{T}^{\alpha_2,\gamma_2})_{q_0- p_0}(\operatorname{Box}^{\operatorname{in}}(y,\epsilon^4,R))\bigcap \operatorname{Box}^{\operatorname{in}}(y,\epsilon^4,R).$}.

Notice that $x_{1,h}^{p_0-q_0}=R^{N_1(x_1,p_0-q_0)}_{\alpha_1}x_{1,h}$ and this implies:
\begin{equation}\label{eq:horizontalDistDiop}
d_H(x_1^{p_0-q_0},x_1)=\min_{m\in\mathbb{Z}}\{|N_1(x_1,p_0-q_0)\alpha_1-m|\}.
\end{equation}

Applying Lemma \ref{lm:hittingTimeEstimates} with $\epsilon=\epsilon^2$, $\epsilon_1\in(0,\epsilon_{bound}]$, $a_1=1+|\gamma_1|+2\epsilon_1$, $a_2=1+|\gamma_2|+2\epsilon_1$ and $a_3=1$. By the first part of the Lemma \ref{lm:hittingTimeEstimates}, we obtain that for any $y\in\bar{E}_{\epsilon}$, any $x\in\operatorname{Box}^{\operatorname{in}}(y,\epsilon^4,R)$ and any $L\in[0,R]$, we have
\begin{equation}\label{eq:NcloseControl}
\begin{aligned}
N_1(y_1,L)-2\leq N_1(x_1,L)\leq N_1(y_1,L)+2.
\end{aligned}
\end{equation}

Recall that $p_0-q_0\geq C_{\alpha_1,\gamma_1}$, definition of $C_{\alpha_1,\gamma_1}$, \eqref{eq:NcloseControl} and equation $(7.1)$ in \cite{KanigowskiWei1} (notice that here we only need to control $N_1(x_1, p_0-q_0)$ and thus (7.1) in \cite{KanigowskiWei1} will be enough), we obtain that
\begin{equation}\label{eq:SpecialDisCon}
\begin{aligned}
&N_1(x_1, p_0-q_0)\geq N_1(y_1, p_0-q_0)-2\geq \frac{p_0-q_0}{\log^5(p_0-q_0)}-2\\
&\geq\frac{C_{\alpha_1,\gamma_1}}{\log^5C_{\alpha_1,\gamma_1}}-2\geq C_{\alpha_1,\gamma_1}^{1-|\gamma_1|}\geq \max\{N_{\alpha_1},C_{\gamma_1},C_0,1\}.
\end{aligned}
\end{equation}

Then \eqref{eq:Ncontrol}, \eqref{eq:SpecialDisCon} and definition of $C_{\alpha_1,\gamma_1}$ imply that
\begin{equation}\label{eq:finalOfDisjointLemma}
\begin{aligned}
\min_{m\in\mathbb{Z}}\{|N_1(x_1,p_0-q_0)\alpha_1-m|\}&\geq\frac{1}{N_1(x_1,p_0-q_0)\log^{10}N_1(x_1, p_0-q_0)}\\
&\geq\frac{1}{N_1(x_1,p_0-q_0)^{1+(|\gamma_1|/2)}}\\
&\geq\frac{1}{[(\epsilon_{bound})^{-1}(p_0-q_0)]^{1+(|\gamma_1|/2)}}\geq\frac{1}{R^{1+(2|\gamma_1|/3)}},
\end{aligned}
\end{equation}
where $\epsilon_{bound}$ is defined in \eqref{def:ebound}.

But by the definition of $x$, the right side of \eqref{eq:horizontalDistDiop} need to be smaller than $\frac{2\epsilon^4}{R^{1+|\gamma_1|}}$, which contradicts to \eqref{eq:finalOfDisjointLemma} and thus we finish the proof.
\end{proof}

\noindent\textbf{Lower bound estimate:}
Fix $\delta>0$, $0<\epsilon<\frac{1}{100}$ and for every $\epsilon_1>0$, let $m\geq\max\{m_{\delta},4\epsilon_{bound}^{-8}\}$, $R\geq \max\{N_{0,\epsilon}, R_{\delta}\}$ and $D=D_{\delta}$, where $N_{0,\epsilon}$ is defined in Lemma \ref{lem:linearControl} and $R_{\delta}$, $D_{\delta}$ is defined in Theorem \ref{lm:mainLemma}. For every $x,y\in D$, if $h(x,y)$ is a $(\epsilon,\mathcal{P}_m,R)-$matching, by Theorem \ref{lm:mainLemma} there exists $t_0\in A(x,y)$ such that
$$
L_H(t_0)\leq R^{-\frac{1}{1-\epsilon_0}}.
$$
In particular, we have
\begin{equation}\label{eq:HorizontalControl}
d_H(x_1^{t_0},y_1^{h(t_0)})\leq R^{-\frac{1}{1-\epsilon_0}},\ \ d_H(x_2^{t_0},y_2^{h(t_0)})\leq R^{-\frac{1}{1-\epsilon_0}}.
\end{equation}

Combining \eqref{eq:HorizontalControl}, Definition \ref{def:outBox} and recall that $t_0\in A(x,y)$, we obtain
\begin{equation}\label{eq:lowerBoundClosePoints}
x^{t_0}\in \operatorname{Box}^{\operatorname{out}}(y^{h(t_0)},\frac{2}{m},R).
\end{equation}
This implies that
\begin{equation}\label{eq:outcontrol1}
x\in \phi_{-t_0}\left(\operatorname{Box}^{\operatorname{out}}(y^{h(t_0)},\frac{2}{m},R)\right).
\end{equation}



Let $\eta=\eta(\gamma_1,\epsilon_0,\epsilon_1)=\frac{1}{(1-\epsilon_0)(1+|\gamma_1|+2\epsilon_1)}$, $R_i=iR^{\eta}$. It is worth to point out that by the choice of $\epsilon_0$ and $\epsilon_1$, we have $\eta<1$. We will show that any $y\in D$ which are $(\frac{1}{100},\mathcal{P}_m,R)-$ matchable satisfy
\begin{equation}\label{eq:outcontrol3}
B_R(y,\epsilon,\mathcal{P}_m)\cap D\subset\bigcup_{-4R\leq p\leq 4R}\bigcup_{i=1}^{R^{1-\eta}}\phi_{-p}\left(\operatorname{Box}_{\sqrt{\frac{2}{m}}}^{\operatorname{out}}(y^{R_i},\sqrt{\frac{2}{m}},R)\right).
\end{equation}

By Definition \ref{def:outBox} and \eqref{eq:outcontrol3}, we obtain
\begin{equation}\label{eq:upperboundBRball}
\begin{aligned}
\tilde{\mu}\left(B_R(y,\epsilon,\mathcal{P}_m)\cap D\right)&\leq8R^{2-\eta}\max_i\{\tilde{\mu}\left(\operatorname{Box}^{\operatorname{out}}_{\sqrt{\frac{2}{m}}}(y^{R_i},\sqrt{\frac{2}{m}},R)\right)\}\\
&\leq8\cdot\frac{100}{m^2}R^{2-\frac{1}{(1-\epsilon_0)(1+|\gamma_1|+2\epsilon_1)}-\frac{2}{1-\epsilon_0}},
\end{aligned}
\end{equation}
where last inequality obtained by applying Lemma \ref{lm:MovingOurBallMeasure} with $\epsilon=\sqrt{\frac{2}{m}}$.


Then define $W(\epsilon_0,\epsilon_1)=2-\frac{1}{(1-\epsilon_0)(1+|\gamma_1|+2\epsilon_1)}-\frac{2}{1-\epsilon_0}$. Thus by Lemma \ref{compef} and \eqref{eq:upperboundBRball} the number of balls needed to cover $1-\epsilon$ portion of space is at least $\frac{m^2}{800}R^{-W(\epsilon_0,\epsilon_1)}$ and therefore we obtain $\beta(\log,\epsilon,\mathcal{P}_m)\geq-W(\epsilon_0,\epsilon_1)$. Recall that the sequence of partitions $\mathcal{P}_m$ is generating thus we have $$e(\phi_t,\log)\geq-W(\epsilon_0,\epsilon_1).$$

Recall that $\epsilon_0=\frac{|\gamma_2|}{2(1+|\gamma_2|)}-4\epsilon_2\frac{2+|\gamma_2|}{(|\gamma_1|-|\gamma_2|)(1+|\gamma_2|)}$ for $\epsilon_2\in(0,\frac{|\gamma_2|(|\gamma_1|-|\gamma_2|)}{16(2+|\gamma_2|)})$, $-W(\epsilon_0,\epsilon_1)$ is increasing in $\epsilon_0$ and decreasing in $\epsilon_1$, thus we have $$e(\phi_t,\log)\geq\lim_{\epsilon_1\to0,\epsilon_2\to0}-W(\epsilon_0,\epsilon_1)=\frac{2+4|\gamma_2|+2|\gamma_1\gamma_2|}{(2+|\gamma_2|)(1+|\gamma_1|)}.$$
So it remains to show \eqref{eq:outcontrol3}.

Recall \eqref{eq:lowerBoundClosePoints}, we obtain that there exists a $t_0\in[0,R]$ such that
$$x^{t_0}\in \operatorname{Box}^{\operatorname{out}}(y^{h(t_0)},\frac{2}{m},R).
$$

Let $R_{i_0}$ minimize $R_i-h(t_0)$ over all $i$ with $R_i-h(t_0)\geq0$. Then we apply Lemma \ref{lem:linearControl} with $\epsilon=\sqrt{\frac{2}{m}}$, $x=x^{t_0}$ and $y=y^{h(t_0)}$ (recall that $\bar{E}_{\epsilon}=(S^1(n')\cap W^1(\delta))\times(S^2(n')\cap W^2(\delta))$ and $x^{t_0}\in(S^1(n')\cap W^1(\delta))\times(S^2(n')\cap W^2(\delta))$ due to Theorem \ref{lm:mainLemma}), we obtain there exists $L'$ with $|L'-(R_{i_0}-h(t_0))|<\frac{12}{m}$ such that:
$$x^{L'+t_0}\in \operatorname{Box}^{\operatorname{out}}_{\sqrt{\frac{2}{m}}}(y^{R_{i_0}},\sqrt{\frac{2}{m}},R).
$$
Notice that $|t_0+L'|\leq|t_0|+|R_{i_0}-h(t_0)|+\frac{12}{m}\leq 3R$ and this gives \eqref{eq:outcontrol3} and thus finishes the proof.


\end{proof}
\section{Proof of Theorem  \ref{lm:mainLemma}}\label{sec:proofOfMainLemma}
\begin{theorem}
\label{lm:mainLemma}
For every $\delta>0$ there exists a set $D=D_{\delta}\subset\mathbb{T}^{f_1}\times\mathbb{T}^{f_2}$ with $\tilde{\mu}(D)>1-\delta$ and $m_{\delta},R_{\delta}\in\mathbb{R}$ such that for every $x=(x_1,x_2),y=(y_1,y_2)\in D$, $m\geq m_{\delta}$, $R\geq R_{\delta}$ and $x,y$ are $(\frac{1}{100},\mathcal{P}_m,R)-$matchable, there exists $t_0\in A(x,y)$\footnote{$A(x,y)$ is defined in Definition \ref{def:ematchable}.} such that $x^{t_0}\in (S^1(n')\cap W^1(\delta))\times(S^2(n')\cap W^2(\delta))$ and
\begin{equation}
L_H(t_0)\leq R^{-\frac{1}{1-\epsilon_0}},
\end{equation}
where $\epsilon_0=\frac{|\gamma_2|}{2(1+|\gamma_2|)}-4\epsilon_2\frac{2+|\gamma_2|}{(|\gamma_1|-|\gamma_2|)(1+|\gamma_2|)}$ and $\epsilon_2\in(0,\frac{|\gamma_2|(|\gamma_1|-|\gamma_2|)}{16(2+|\gamma_2|)})$ is a fixed number.
\end{theorem}

Some parts of the proof of Theorem \ref{lm:mainLemma} is similar to the Proposition $3.1$ in \cite{KanigowskiWei1}, we provide a complete proof here as we get \emph{much sharper} estimates by more precise inequalities (see Section \ref{sec:SharpEstimate} for more details) and we also directly conduct our proof in the case of flows
instead of time-one maps of flows as in \cite{KanigowskiWei1}.

For any $m>0$, $R\geq1$ and $j\in\mathbb{N}$, suppose $x,y\in\mathbb{T}^{f_1}\times\mathbb{T}^{f_2}$ are $(\frac{1}{100},\mathcal{P}_m,R)-$matchable, define a set $A_j^{R,m}(x,y)$ as:
$$A_j^{R,m}(x,y)=\{t\in A(x,y):L(t)<\frac{2}{m}\text{ and }2^{-j-1}<L_H(t)\leq2^{-j}\}.$$

The meaning of the set $A_j^{R,m}(x,y)$ can be explained as following: we decompose all the matching pairs into different sets based their horizontal distance. The following two graphs provide an intuitive explanation about what is $A_j^{R,m}(x,y)$, where different colors represent different $A_j^{R,m}(x,y)$:
\begin{figure}[H]
 \centering
 \scalebox{0.6}
 {
 \begin{tikzpicture}[scale=5]
	 \tikzstyle{vertex}=[circle,minimum size=20pt,inner sep=0pt]
	 \tikzstyle{selected vertex} = [vertex, fill=red!24]
	 \tikzstyle{edge1} = [draw,line width=5pt,-,red!50]
     \tikzstyle{edge2} = [draw,line width=5pt,-,green!50]
     \tikzstyle{edge3} = [draw,line width=5pt,-,blue!50]
     \tikzstyle{edge4} = [draw,line width=5pt,-,brown!50]

	 \tikzstyle{edge} = [draw,thick,-,black]
	 \node[vertex] (v00) at (0,0) {$x$};
	 \node[vertex] (v01) at (0.6,0) {$\phi_{t_1}x$};
	 \node[vertex] (v02) at (1.0,0) {$\phi_{t_2}x$};
	 \node[vertex] (v03) at (1.2,0) {$\phi_{t_3}x$};
	 \node[vertex] (v04) at (1.5,0) {$\phi_{t_4}x$};
	 \node[vertex] (v05) at (2.1,0) {$\phi_{t_5}x$};
	 \node[vertex] (v06) at (2.7,0) {$\phi_{t_6}x$};
	 \node[vertex] (v07) at (3.3,0) {$\phi_{t_7}x$};
     \node[vertex] (v10) at (0,0.5) {$\phi_{h(t_1)}y$};
	 \node[vertex] (v11) at (0.4,0.5) {$\phi_{h(t_2)}y$};
	 \node[vertex] (v12) at (1.2,0.5) {$\phi_{h(t_3)}y$};
	 \node[vertex] (v13) at (1.8,0.5) {$\phi_{h(t_4)}y$};
	 \node[vertex] (v14) at (2.1,0.5) {$\phi_{h(t_5)}y$};
	 \node[vertex] (v15) at (2.4,0.5) {$\phi_{h(t_6)}y$};
	 \node[vertex] (v16) at (3.3,0.5) {$\phi_{h(t_7)}y$};
	
	 \draw[edge] (v00)--(v01)--(v02)--(v03)--(v04)--(v05)--(v06)--(v07);
	 \draw[edge] (v10)--(v11)--(v12)--(v13)--(v14)--(v15)--(v16);
     \draw[edge4] (v10)--(v01);
     \draw[edge4] (v11)--(v02);
     \draw[edge4] (v12)--(v03);
     \draw[edge4] (v13)--(v04);
     \draw[edge4] (v14)--(v05);
     \draw[edge4] (v15)--(v06);
     \draw[edge4] (v16)--(v07);

 \end{tikzpicture}
 }
 \caption{Original Matching}
 \end{figure}
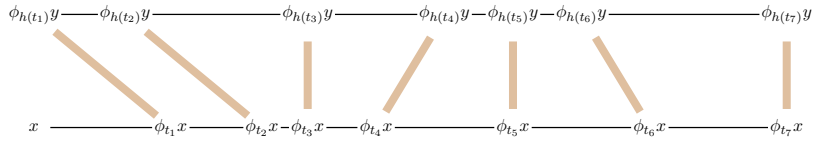
 \begin{figure}[H]
 \centering
 \scalebox{0.6}
 {
 \begin{tikzpicture}[scale=5]
	 \tikzstyle{vertex}=[circle,minimum size=20pt,inner sep=0pt]
	 \tikzstyle{selected vertex} = [vertex, fill=red!24]
	 \tikzstyle{edge1} = [draw,line width=5pt,-,red!50]
     \tikzstyle{edge2} = [draw,line width=5pt,-,green!50]
     \tikzstyle{edge3} = [draw,line width=5pt,-,blue!50]
     \tikzstyle{edge4} = [draw,line width=5pt,-,brown!50]

	 \tikzstyle{edge} = [draw,thick,-,black]
	 \node[vertex] (v00) at (0,0) {$x$};
	 \node[vertex] (v01) at (0.6,0) {$\phi_{t_1}x$};
	 \node[vertex] (v02) at (1.0,0) {$\phi_{t_2}x$};
	 \node[vertex] (v03) at (1.2,0) {$\phi_{t_3}x$};
	 \node[vertex] (v04) at (1.5,0) {$\phi_{t_4}x$};
	 \node[vertex] (v05) at (2.1,0) {$\phi_{t_5}x$};
	 \node[vertex] (v06) at (2.7,0) {$\phi_{t_6}x$};
	 \node[vertex] (v07) at (3.3,0) {$\phi_{t_7}x$};
     \node[vertex] (v10) at (0,0.5) {$\phi_{h(t_1)}y$};
	 \node[vertex] (v11) at (0.4,0.5) {$\phi_{h(t_2)}y$};
	 \node[vertex] (v12) at (1.2,0.5) {$\phi_{h(t_3)}y$};
	 \node[vertex] (v13) at (1.8,0.5) {$\phi_{h(t_4)}y$};
	 \node[vertex] (v14) at (2.1,0.5) {$\phi_{h(t_5)}y$};
	 \node[vertex] (v15) at (2.4,0.5) {$\phi_{h(t_6)}y$};
	 \node[vertex] (v16) at (3.3,0.5) {$\phi_{h(t_7)}y$};
	
	 \draw[edge] (v00)--(v01)--(v02)--(v03)--(v04)--(v05)--(v06)--(v07);
	 \draw[edge] (v10)--(v11)--(v12)--(v13)--(v14)--(v15)--(v16);
     \draw[edge1] (v10)--(v01);
     \draw[edge1] (v11)--(v02);
     \draw[edge2] (v12)--(v03);
     \draw[edge3] (v13)--(v04);
     \draw[edge2] (v14)--(v05);
     \draw[edge3] (v15)--(v06);
     \draw[edge2] (v16)--(v07);

 \end{tikzpicture}
 }
 \caption{Partition of Matching Arrows based on set $A_j^{R,m}$}
 \end{figure}
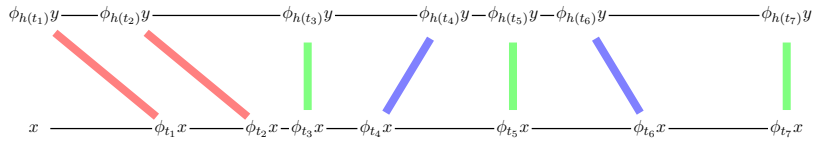

We will use the following Lemma to prove Theorem  \ref{lm:mainLemma}, where, we recall, $K_m^i$ is defined in the end of Section \ref{sec:Notations} as the suspension space $\mathbb{T}^{f_i}$ cut at height $2^n$ for $i=1,2$:
\begin{lemma}\label{lm:secondMain}
For every $\delta>0$ there exists $R_{\delta},m_{\delta}>0$ and a set $D=D_{\delta}\subset\mathbb{T}^{f_1}\times\mathbb{T}^{f_2}$ with $(\mu^{f_1}\times\mu^{f_2})(D)>1-\delta$ such that for all $m\geq m_{\delta}$, $R\geq R_{\delta}$ and $x,y\in D$ that are $(\frac{1}{100},\mathcal{P}_m,R)$-matchable, there exists $W_R(x,y)\subset A(x,y)$ such that
\begin{itemize}
  \item[\textbf{(a)}] $\bar{\lambda}(W_R(x,y))\geq\frac{9}{10}R$ and for any $t\in W_R(x,y)$, $x^t\in(S^1(n')\cap W^1(\delta))\times(S^2(n')\cap W^2(\delta))$,
  \item[\textbf{(b)}] for every $p\in W_R(x,y)$, we have $x^p,y^{h(p)}\in K_m^1\times K_m^2$,
  \item[\textbf{(c)}] for every $j\in\mathbb{N}$ satisfying $2^j\leq R^{\frac{1}{1-\epsilon_0}}$ we have
$\bar{\lambda}(A_j^{R,m}(x,y)\cap W_R(x,y))\leq\frac{R}{j^{1.4}}.$
\end{itemize}
\end{lemma}

The proof of Lemma \ref{lm:secondMain} is contained in the next section. Let's first give a conditional proof of Theorem \ref{lm:mainLemma} based on Lemma \ref{lm:secondMain}.

\begin{proof}[Proof of Theorem \ref{lm:mainLemma}]
Fix $\delta>0$, $R\geq R_{\delta}$, $m\geq m_{\delta}$ and $x,y\in D$ are $(\frac{1}{100},\mathcal{P}_m,R)$-matchable with the set $A(x,y)$ and the matching function $h$. By Lemma \ref{lm:secondMain} \textbf{(b)}, the definition of $K_m^1$ and $K_m^2$ and the definition of $A_j^{R,m}$, we obtain
$$W_R(x,y)\subset\bigcup_{j\in\mathbb{N}}A_j^{R,m}.$$
Also notice that by the definition of $A_j^{R,m}$, for $j\leq\log \frac{m}{8}$ we have
\begin{equation}\label{eq:AemptySet}
A_j^{R,m}(x,y)=\emptyset.
\end{equation}
Thus by \textbf{(a)} and \textbf{(b)} of Lemma \ref{lm:secondMain}, we have
\begin{equation}\label{eq:SumSmall}
\begin{aligned}
\frac{1}{100}&>f_R(x,y,\mathcal{P}_m)\\
&\geq 1-\frac{\bar{\lambda}(W_R(x,y)^c)}{R}-\frac{1}{R}\sum_{j\geq0}\bar{\lambda}(A_j^{R,m}(x,y)\cap W_R(x,y))\\
&\geq\frac{9}{10}-\frac{1}{R}\sum_{j\geq0}\bar{\lambda}(A_j^{R,m}(x,y)\cap W_R(x,y)),
\end{aligned}
\end{equation}
where the second inequality follows by considering the possibility such that $f_R(x,y,\mathcal{P}_m)$ as small as we can, i.e. all the elements outside $W_R(x,y)$ are good matching time moments and thus do not contribute to $f_R(x,y,\mathcal{P}_m)$.

Let $j_R$ be such that $2^{j_R}\leq R^{\frac{1}{1-\epsilon_0}}<2^{j_R+1}$. Then by \eqref{eq:AemptySet}, Lemma \ref{lm:secondMain} \textbf{(c)} we obtain
\begin{equation}\label{eq:SumSmall1}
\frac{1}{R}\sum_{0\leq j\leq j_R}\bar{\lambda}(A_j^{R,m}(x,y)\cap W_R(x,y))=\frac{1}{R}\sum_{\log\frac{m}{8}\leq j\leq j_R}\bar{\lambda}(A_j^{R,m}(x,y)\cap W_R(x,y))\leq\frac{1}{1000}
\end{equation}
by picking $m$ large enough (as $\sum_{j}\frac{1}{j^{1.4}}<\infty$). Thus combining \eqref{eq:SumSmall} and \eqref{eq:SumSmall1}, there exists $j_1> j_R$ such that $A_{j_1}^{R,m}(x,y)\cap W_R(x,y)\neq\emptyset$. Finally recalling the definition of $A_{j_1}^{R,m}(x,y)$ and $W_R(x,y)$, we obtain that there exists $t_0\in A(x,y)$ such that $x^{t_0}\in(S^1(n')\cap W^1(\delta))\times(S^2(n')\cap W^2(\delta))$ and $$L_H(t_0)\leq 2^{-j_1}\leq2^{-(j_R+1)}<R^{-\frac{1}{1-\epsilon_0}},$$
which finishes the proof of Theorem \ref{lm:mainLemma}.
\end{proof}

\section{Proof of Lemma \ref{lm:secondMain}}\label{sec:proofOfSecondLemma}
The proof of Lemma \ref{lm:secondMain} will follow from the following lemma, where $B(\cdot,\cdot)$ is defined in Definition \ref{def:MatchingBall}:
\begin{lemma}\label{lm:thirdMain}
For every $\delta>0$, there exists $R_{\delta},m_{\delta}>0$ and $D=D_{\delta}\subset\mathbb{T}^{f_1}\times\mathbb{T}^{f_2}$ with $(\mu^{f_1}\times\mu^{f_2})(D)>1-\delta$ such that for any $R\geq R_{\delta}$, $m\geq m_{\delta}$ and every $x,y\in D$ that are $(\frac{1}{100},\mathcal{P}_m,R)-$matchable, there exists $W_R(x,y)\subset A(x,y)$ such that \textbf{(a)} and \textbf{(b)} of Lemma \ref{lm:secondMain} hold and for every $j\in\mathbb{N}$ satisfying $2^j\leq R^{\frac{1}{1-\epsilon_0}}$ we have
\begin{itemize}
  \item[(1)] for $w\in W_R(x,y)$ and denote $R_w^{-1}\doteq L_H(w)$, if $L_H(w)<\frac{2}{m}$, then for every $s\in B(w,\frac{R_w}{\log^5R_w})$, either
\begin{equation}\label{eq:HorizontalIsom}
d_H(x^s_i,y^{h(s)}_i)=d_H(x^w_i,y^{h(w)}_i)\text{ for }i=1,2
\end{equation}
or
\begin{equation}\label{eq:badsituation}
L_H(s)\geq 100L_H(w);
\end{equation}
  \item[(2)] for every $w\in W_R(x,y)$ and denote $R_w^{-1}\doteq L_H(w)$, if $L_H(w)<\frac{2}{m}$, we have at least one of the following inequalities:
\begin{equation}\label{eq:ControlInequal1}
\begin{aligned}
\bar{\lambda}(\{t\in B(w,R_w^{\frac{1}{1+|\gamma_2|}+\epsilon_0}):\eqref{eq:HorizontalIsom}\text{ holds and }L(t)<\frac{1}{4}\})<\frac{R_w^{\frac{1}{1+|\gamma_2|}+\epsilon_0}}{\log^2 R_w}
\end{aligned}
\end{equation}
or
\begin{equation}\label{eq:ControlInequal2}
\begin{aligned}
\bar{\lambda}(\{t\in B(w,R_w^{1-\epsilon_0}):\eqref{eq:HorizontalIsom}\text{ holds and }L(t)<\frac{1}{4}\})<\frac{R_w^{1-\epsilon_0}}{\log^2 R_w}.
\end{aligned}
\end{equation}
\end{itemize}
\end{lemma}

\begin{remark}
The meaning of Lemma \ref{lm:thirdMain} should be divided into two parts: the first part shows that for $x^w$ and $y^{h(w)}$, if $s$ is in a neighborhood of $w$, then either the horizontal distance between $x^s$ and $y^{h(s)}$ is same as horizontal distance between $x^w$ and $y^{h(w)}$ on each coordinates or the horizontal distance at time $s$ is extremely larger than the horizontal distance at time $w$, which leading to that $w$ and $s$ cannot stay in same $A_j^{R,m}(x,y)$. The second part of the Lemma \ref{lm:thirdMain} shows that for the set consisted of ``good'' matching time moments, i.e. have same horizontal distance as initial matching pair, its Lebesgue measure is always relative small with respect to Lebesgue measure of the full neighborhood.
\end{remark}

We will show how Lemma \ref{lm:thirdMain} give the proof of Lemma \ref{lm:secondMain} now.
\begin{proof}[Proof of Lemma \ref{lm:secondMain}]
Fix $\delta>0$, we only need to prove \textbf{(c)} of Lemma \ref{lm:secondMain} based on $(1)$ and $(2)$ of Lemma \ref{lm:thirdMain}.

If $\frac{m}{8}\geq R^{\frac{1}{1-\epsilon_0}}$, notice that for $j\leq\log \frac{m}{8}$, by \eqref{eq:AemptySet} we get $\bar{\lambda}(A_j^{R,m}(x,y))=0$ and thus the proof of Lemma \ref{lm:secondMain} is finished. From now on, we only need to deal with the situation $\log\frac{m}{8}<j\leq \frac{1}{1-\epsilon_0}\log R$.

For any $\log\frac{m}{8}<j\leq \frac{1}{1-\epsilon_0}\log R$, divide the interval $[0,R]$ into several disjoint intervals $I_1,\ldots,I_k$ of length $l_j^1:=\frac{(2^j)^{\frac{1}{1+|\gamma_2|}+\epsilon_0}}{\sqrt{\log2^j}}$ or $l_j^2:=\frac{(2^j)^{1-\epsilon_0}}{\sqrt{\log2^j}}$ by the following inductive procedure:
\begin{itemize}
  \item[1.] Fix the smallest element $w_1\in W_R(x,y)\cap A_j^{R,m}(x,y)$, if $w_1$ satisfies \eqref{eq:ControlInequal1} of Lemma \ref{lm:thirdMain}, let $I_1$ be an interval with left endpoint $w_1$ and length $l_j^1$; if $w_1$ satisfies \eqref{eq:ControlInequal2} of Lemma \ref{lm:thirdMain}, let $I_1$ be an interval with left endpoint $w_1$ with length $l_j^2$;
  \item[2.] Then inductively for $u>1$, let $w_u$ be the smallest element in $(W_R(x,y)\cap A_j^{R,m}(x,y))\setminus(I_1\cup\ldots\cup I_{u-1})$. Based on whether $w_u$ satisfies \eqref{eq:ControlInequal1} or \eqref{eq:ControlInequal2} of Lemma \ref{lm:thirdMain}, we let $I_u$ be the interval with left endpoint $w_u$ and length $l_j^1$ or $l_j^2$;
  \item[3.] We continue this procedure until we cover $W_R(x,y)\cap A_j^{R,m}(x,y)$ and define these intervals as $I_1,\ldots,I_k$.
\end{itemize}
Notice that since $2^j\leq R^{\frac{1}{1-\epsilon_0}}$ and $\frac{1}{1+|\gamma_2|}+\epsilon_0<1-\epsilon_0$, thus it follows $k>1$. Moreover, we have the following estimate of the number of the intervals based on definition:
\begin{equation}\label{eq:NumberOfIntervals}
\operatorname{Card}(\{i\in\{1,\ldots,k\}:\bar{\lambda}(I_i)=l_j^s\})\leq\frac{R}{l_j^s}+2 \text{ for }s=1,2.
\end{equation}

Recall the definition of $A_j^{R,m}(x,y)$ and  $R_{w_i}^{-1}=L_H(w_i)\in(2^{-j-1},2^{-j}]$, we have $$\frac{R_{w_i}}{\log^5R_{w_i}}\geq\bar{\lambda}(I_i)$$ for every $i\in\{1,\ldots,k\}$ (as $j>\log\frac{m}{8}$ and we can pick $m_\delta$ large enough). Thus by $(1)$ of Lemma \ref{lm:thirdMain} and the definition of $A_j^{R,m}(x,y)$ we obtain
\begin{equation}\label{eq:CombinoticsContain}
A_j^{R,m}(x,y)\cap W_R(x,y)\cap I_i\subset\{t\in I_i:\eqref{eq:HorizontalIsom}\text{ holds, }L(t)<\frac{2}{m}\}.
\end{equation}
It is worth to point that on the right side of \eqref{eq:CombinoticsContain} the reason that we do not need to consider the situation that \eqref{eq:badsituation} holds is that definition of $A_j^{R,m}(x,y)$ forbids this.

By $(2)$ of Lemma \ref{lm:thirdMain}, we have one of the following holds:
\begin{equation}\label{eq:IntervalControl1}
\bar{\lambda}(\{t\in I_i:\eqref{eq:HorizontalIsom}\text{ holds, }L(t)<\frac{2}{m}\})\leq\frac{(2^j)^{\frac{1}{1+|\gamma_2|}+\epsilon_0}}{\log^22^{j+1}},
\end{equation}
or
\begin{equation}\label{eq:IntervalControl2}
\bar{\lambda}(\{t\in I_i:\eqref{eq:HorizontalIsom}\text{ holds, }L(t)<\frac{2}{m}\})\leq\frac{(2^j)^{1-\epsilon_0}}{\log^22^{j+1}}.
\end{equation}

Now summing over $i\in\{1,\ldots,k\}$ and combining with \eqref{eq:NumberOfIntervals}, \eqref{eq:CombinoticsContain}, \eqref{eq:IntervalControl1} and \eqref{eq:IntervalControl2}, we obtain
\begin{equation}\label{eq:finalIntervalControl}
\begin{aligned}
\bar{\lambda}(A_j^{R,m}(x,y)\cap W_R(x,y))&\leq\frac{R}{l_j^2}\frac{(2^j)^{1-\epsilon_0}}{\log^22^{j+1}}+\frac{R}{l_j^1}\frac{(2^j)^{\frac{1}{1+|\gamma_2|}+\epsilon_0}}{\log^22^{j+1}}\leq\frac{R}{\log^{3/2}2^j}+\frac{R}{\log^{3/2}2^j}\\
&\leq\frac{R}{j^{1.4}},
\end{aligned}
\end{equation}
where the second inequality follows by the definition of $l_j^1$ and $l_j^2$ and last inequality due to $j>\log\frac{m}{8}$ and we can pick $m_{\delta}$ large enough. Thus \eqref{eq:finalIntervalControl} finish the proof of Lemma \ref{lm:secondMain}.

\end{proof}
\section{Proof of Lemma \ref{lm:thirdMain}}\label{sec:proofOfThirdLemma}

For $\delta>0$, let $n'$ as in Lemma \ref{lm:derivativeControl}, define set $F$ as:
\begin{equation}
F=F_{\delta}=\prod_{i=1,2}(S^i(n')\cap W^i(\delta)\cap\{x_i\in\mathbb{T}^{f_i}:f_i(x_{i,h})<\delta^{-\frac{3}{1-|\gamma_i|}}\}).
\end{equation}
It follows from Lemma \ref{lm:derivativeControl} that $\tilde{\mu}(F)\geq1-\delta^2$.

\emph{Construction of $D_{\delta}$, $W_R(x,y)$, $R_{\delta}$ and $m_{\delta}$:} By applying ergodic theorem to $\phi_t$ and set $F$, we know that there exists a set $D_{\delta}\subset\mathbb{T}^{f_1}\times\mathbb{T}^{f_2}$ with $\tilde{\mu}(D_{\delta})>1-\delta$ and $R_{\delta}\in\mathbb{R}$ such that for all $x\in D_{\delta}$ and $R\geq R_{\delta}$ we have
$$\bar{\lambda}(\{t\in[0,R]:x^t\in F\})\geq(1-\delta)R.$$

For $x\in D_{\delta}$ define $W_R(x)$ as follows:
\begin{equation}\label{eq:WRxy}
W_R(x)=\{t\in[0,R]:x^t\in F\}.
\end{equation}
By taking $\delta$ small enough, we obtain $\bar{\lambda}(W_R(x))\geq\frac{99R}{100}$ and
\begin{equation}\label{eq:Kdelta}
x^t\in K_{\delta^{-1}}^1\times K_{\delta^{-1}}^2 \text{ for every }t\in W_R(x).
\end{equation}
Then define $W_R(x,y)$ as following:
\begin{equation}
W_R(x,y)=W_R(x)\cap W_R(y)\cap A(x,y).
\end{equation}

Now the Lemma \ref{lm:secondMain} \textbf{(a)} and \textbf{(b)} follow from the definition of $D_{\delta}$ and definition of $W_R(x,y)$ by defining $m_{\delta}=\delta^{-1}$. Thus the only remaining parts of Lemma \ref{lm:thirdMain} is $(1)$ and $(2)$.

In fact, Lemma \ref{lm:thirdMain} $(1)$ follows from the Lemma $6.1$ of \cite{KanigowskiWei1} by setting $t=s-w$ and $z=x_1$, $z'=x_2$.
\subsection{Proof of (2) of Lemma \ref{lm:thirdMain}}\label{sec:SharpEstimate}
Fix $w\in A(x,y)$, we claim that for every $t$ such that \eqref{eq:HorizontalIsom} holds and $L(t)<\frac{1}{4}$, we have
\begin{equation}\label{eq:BirkhoffDifference}
|(f_1^{(N_1(x_1,t-w))}(x_{1,h})-f_1^{(N_1(x_1,t-w))}(y_{1,h}))-(f_2^{(N_2(x_2,t-w))}(x_{2,h})-f_2^{(N_2(x_2,t-w))}(y_{2,h}))|\leq\frac{1}{2}.
\end{equation}

In fact, it follows from \eqref{eq:HorizontalIsom} that the action of $(\mathscr{T}^{\alpha_1,\gamma_1}\times\mathscr{T}^{\alpha_2,\gamma_2})_t$ on the circle $\mathbb{T}$ is isometric, thus by the definition of special flow and $t,w\in A(x,y)$, we have for $i=1,2$:
\begin{equation}\label{eq:BirkhoffDifferenceExp}
|[(t-w)-f_i^{(N_i(x_i,t-w))}(x_{i,h})]-[(h(t)-h(w))-f_i^{(N_i(x_i,t-w))}(y_{i,h})]|<\frac{1}{4}.
\end{equation}
Then \eqref{eq:BirkhoffDifference} follows from \eqref{eq:BirkhoffDifferenceExp} and triangle equality.

Moreover, notice $w\in W_R(x,y)$ implies $x^w\in F$ and thus $x_i^w\in S_i(n')$ for $i=1,2$. Then we obtain that $f_i^{(N_i(x_i,t-w))}$ is differentiable on $[x_{i,h},y_{i,h}]$ by Lemma \ref{lm:differentialF} for $i=1,2$. Thus \eqref{eq:BirkhoffDifference} is equivalent to
\begin{equation}\label{eq:BirkhoffDifferenceDer}
|f_1'^{(N_1(x_1,t-w))}(\theta_{1})(x_{1,h}-y_{1,h})-f_2'^{(N_2(x_2,t-w))}(\theta_2)(x_{2,h}-y_{2,h})|\leq\frac{1}{2},
\end{equation}
where $\theta_1\in[x_{1,h},y_{1,h}]$ and $\theta_2\in[x_{2,h},y_{2,h}]$.

From now on, we will assume that Lemma \ref{lm:thirdMain} $(2)$ does not hold. For any $\epsilon_2\in(0,\frac{|\gamma_2|(|\gamma_1|-|\gamma_2|)}{16(2+|\gamma_2|)})$, we apply Lemma \ref{lm:derivativeControl} to $x^w$ with time interval $[0,R_w]$ and recall
\begin{equation}\label{eq:ChoiceOfEpsilon0}
\epsilon_0=\frac{|\gamma_2|}{2(1+|\gamma_2|)}-4\epsilon_2\frac{2+|\gamma_2|}{(|\gamma_1|-|\gamma_2|)(1+|\gamma_2|)}.
\end{equation}

The next claim is crucial to complete our proof, we postpone its proof to the end of the section:
\begin{claim}\label{claim}
There exists $t_1,t_2>w$ such that:
\begin{itemize}
  \item[(i).] $t_1\in B(w,R_w^{\frac{1}{1+|\gamma_2|}+\epsilon_0})\setminus B(w,\frac{1}{2}R_w^{\frac{1}{1+|\gamma_2|}+\epsilon_0}\log^{-2} R_w)$ and $t_1-w\in G_{R_w^{1/(1+|\gamma_2|)+\epsilon_0}}$, where $G_{R_w^{1/(1+|\gamma_2|)+\epsilon_0}}$ is defined in Lemma \ref{lm:derivativeControl};
  \item[(ii).] $t_2\in B(w,R_w^{1-\epsilon_0})\setminus B(w,\frac{1}{2}R_w^{1-\epsilon_0}\log^{-2} R_w)$ and $t_2-w\in G_{R_w^{1-\epsilon_0}}$.
\end{itemize}
\end{claim}

Recall that $R_w^{-1}=\max\{d_H(x_1^w,y_1^{h(w)}),d_H(x_2^{w},y_2^{h(w)})\}$, if $R_w^{-1}=d_H(x_1^w,y_1^{h(w)})$, by \eqref{eq:BirkhoffDifferenceDer} and Claim \ref{claim}, we have:
\begin{equation}\label{eq:contradict1}
\begin{aligned}
\frac{|f_1'^{(N_1(x_1^w,t_2-w))}(\theta_1)|\|x_{1,h}^w-y_{1,h}^{h(w)}\|-\frac{1}{2}}{|f_2'^{(N_2(x_2^w,t_2-w))}(\theta_2)|}&\leq\|x_{2,h}^{w}-y_{2,h}^{h(w)}\|\\&\leq\frac{|f_1'^{(N_1(x_1^w,t_1-w))}(\theta_1)|\|x_{1,h}^w-y_{1,h}^{h(w)}\|+\frac{1}{2}}{|f_2'^{(N_2(x_2^w,t_1-w))}(\theta_2)|}.
\end{aligned}
\end{equation}
By applying Lemma \ref{lm:derivativeControl} and Claim \ref{claim} $(i)$ for $x_1^w$, $x_2^w$ and $t_1-w$; then we apply Lemma \ref{lm:derivativeControl} and Claim \ref{claim} $(ii)$ for $x_1^w$, $x_2^w$, $t_2-w$, \eqref{eq:contradict1} imply:
\begin{equation}\label{eq:middleContradicts1}
\frac{\frac{1}{(2\log^2 R_w)^{1+|\gamma_1|-\epsilon_2}}R_w^{(1-\epsilon_0)(1+|\gamma_1|-\epsilon_2)-1}-\frac{1}{2}}{R_w^{(1-\epsilon_0)(1+|\gamma_2|+\epsilon_2)}}\leq\frac{R_w^{(\frac{1}{1+|\gamma_2|}+\epsilon_0)(1+|\gamma_1|+\epsilon_2)-1}+\frac{1}{2}}{\frac{1}{(2\log^2R_w)^{1+|\gamma_2|-\epsilon_2}}
R_w^{(\frac{1}{1+|\gamma_2|}+\epsilon_0)(1+|\gamma_2|-\epsilon_2)}}.
\end{equation}

By enlarging $m_\delta$ if necessary, we can guarantee that $R_w$ large enough and thus \eqref{eq:middleContradicts1} implies:
\begin{equation}\label{eq:restrictionChoice1}
\frac{R_w^{(1-\epsilon_0)(1+|\gamma_1|-2\epsilon_2)-1}-\frac{1}{2}}{R_w^{(1-\epsilon_0)(1+|\gamma_2|+\epsilon_2)}}\leq\frac{R_w^{(\frac{1}{1+|\gamma_2|}+\epsilon_0)(1+|\gamma_1|+\epsilon_2)-1}+\frac{1}{2}}{
R_w^{(\frac{1}{1+|\gamma_2|}+\epsilon_0)(1+|\gamma_2|-2\epsilon_2)}},
\end{equation}
which contradicts to the choice of $\epsilon_0$ in \eqref{eq:ChoiceOfEpsilon0} and thus we finish the proof in this case.

If $R_w^{-1}=d_H(x_2^w,y_2^{h(w)})$, then by \eqref{eq:BirkhoffDifferenceDer} and Claim \ref{claim}, we have:
\begin{equation}\label{eq:contradict2}
\begin{aligned}
\frac{|f_2'^{(N_2(x_2^w,t_1-w))}(\theta_2)|\|x_{2,h}^w-y_{2,h}^{h(w)}\|-\frac{1}{2}}{|f_1'^{(N_1(x_1^w,t_1-w))}(\theta_1)|}&\leq\|x^w_{1,h}-y^{h(w)}_{1,h}\|\\&\leq\frac{|f_2'^{(N_2(x^w_2,t_2-w))}(\theta_2)|\|x_{2,h}^w-y_{2,h}^{h(w)}\|+\frac{1}{2}}{|f_1'^{(N_1(x_1^w,t_2-w))}(\theta_1)|}.
\end{aligned}
\end{equation}
Again, by applying Lemma \ref{lm:derivativeControl} and Claim \ref{claim} $(i)$ for $x_1^w$, $x_2^w$ and $t_1-w$; then we apply Lemma \ref{lm:derivativeControl} and Claim \ref{claim} $(ii)$ for $x_1^w$, $x_2^w$, $t_2-w$, \eqref{eq:contradict2} imply:
\begin{equation}\label{eq:middleContradicts2}
\frac{\frac{1}{(2\log^2 R_w)^{1+|\gamma_2|-\epsilon_2}}R_w^{(\frac{1}{1+|\gamma_2|}+\epsilon_0)(1+|\gamma_2|-\epsilon_2)-1}-\frac{1}{2}}{R_w^{(\frac{1}{1+|\gamma_2|}+\epsilon_0)(1+|\gamma_1|+\epsilon_2)}}\leq\frac{R_w^{(1-\epsilon_0)(1+|\gamma_2|+\epsilon_2)-1}+\frac{1}{2}}{\frac{1}{(2\log^2R_w)^{1+|\gamma_1|-\epsilon_2}}R_w^{(1-\epsilon_0)(1+|\gamma_1|-\epsilon_2)}}.
\end{equation}

By enlarging $m_\delta$ if necessary, we can guarantee that $R_w$ large enough and thus \eqref{eq:middleContradicts2} implies:
\begin{equation}\label{eq:restrictionChoice2}
\frac{R_w^{(\frac{1}{1+|\gamma_2|}+\epsilon_0)(1+|\gamma_2|-2\epsilon_2)-1}-\frac{1}{2}}{R_w^{(\frac{1}{1+|\gamma_2|}+\epsilon_0)(1+|\gamma_1|+\epsilon_2)}}\leq\frac{R_w^{(1-\epsilon_0)(1+|\gamma_2|+\epsilon_2)-1}+\frac{1}{2}}{R_w^{(1-\epsilon_0)(1+|\gamma_1|-2\epsilon_2)}},
\end{equation}
which contradicts to the choice of $\epsilon_0$ in \eqref{eq:ChoiceOfEpsilon0} and thus we finish the proof of Lemma \ref{lm:thirdMain} $(2)$.

\begin{proof}[Proof of Claim \ref{claim}]
We will prove $(i)$ of the Claim \ref{claim} and $(ii)$ will follow the same lines. As we assume that $(2)$ of Lemma \ref{lm:thirdMain} does not hold, thus we know that \eqref{eq:ControlInequal1} does not hold. Thus for $\mathscr{B}_w=B(w,R_w^{\frac{1}{1+|\gamma_2|}+\epsilon_0})\setminus B(w,\frac{1}{2}R_w^{\frac{1}{1+|\gamma_2|}+\epsilon_0}\log^{-2} R_w)$, we have
\begin{equation}\label{eq:measureOfGood1}
\bar{\lambda}(\mathscr{B}_w)\geq\frac{1}{2}R_w^{\frac{1}{1+|\gamma_2|}+\epsilon_0}\log^{-2} R_w.
\end{equation}

Let $S_w=R_w^{\frac{1}{1+|\gamma_2|}+\epsilon_0}$ and by using Lemma \ref{lm:derivativeControl} for $\epsilon_2$, $S_w$ and $x^w$, we obtain that for $G_{S_w}\subset[0,S_w]$: \begin{equation}\label{eq:measureOfGood2}
\bar{\lambda}(G_{S_w})\geq S_w(1-4\log^{-3} S_w).
\end{equation}

Combining \eqref{eq:measureOfGood1}, \eqref{eq:measureOfGood2},  $\frac{1}{2}R_w^{\frac{1}{1+|\gamma_2|}+\epsilon_0}\log^{-2} R_w\gg4S_w\log^{-3} S_w$ and $$\{w+t:t\in G_{S_w}\}\subset[w,w+S_w], \mathscr{B}_w\subset[w,w+S_w],$$ we obtain that
\begin{equation}
\mathscr{B}_w\cap \{w+t:t\in G_{S_w}\}\neq\emptyset.
\end{equation}

Finally by picking $t_1$ as any element belongs to $\mathscr{B}_w\cap \{w+t:t\in G_{S_w}\}$ and this gives the proof of Claim \ref{claim} $(i)$.
\end{proof}





\end{document}